\def\version{26/12/2022 -- version 3
\hfill\href{https://arxiv.org/abs/2212.00437}{arXiv:2212.00437}
}

\documentclass[11pt,a4paper]{amsart}

\usepackage{fullpage}

\usepackage{amssymb,amscd,amsxtra}
\usepackage{autobreak}

\usepackage[alphabetic,lite]{amsrefs}

\usepackage[autobold]{mathfixs}

\usepackage[scr,scaled=1.1]{rsfso}

\usepackage{leftidx}

\usepackage[clockwise]{rotating}

\usepackage[hypertexnames=false]{hyperref}

\usepackage{xkvltxp}
\usepackage[nomargin,inline]{fixme}
\fxusetheme{color}
\definecolor{fxnote}{rgb}{0.0000,0.6000,0.0000}

\usepackage{mathtools}
\usepackage{relsize}

\usepackage{tikz}
\usetikzlibrary{arrows,calc}
\usetikzlibrary{commutative-diagrams}
\usetikzlibrary{cd}
\usepackage{tikz-network}

\usepackage[all,poly,web,color,matrix,arrow]{xy}
\newdir{ >}{{}*!/-8pt/@{>}}

\usepackage{braket}

\usepackage{pigpen}

\renewcommand{\thefootnote}{\fnsymbol{footnote}}
\long\def\symbolfootnote[#1]#2{\begingroup%
\def\thefootnote{\fnsymbol{footnote}}\footnote[#1]{#2}\endgroup}

\newtheorem{thm}{Theorem}[section]
\newtheorem{lem}[thm]{Lemma}
\newtheorem{prop}[thm]{Proposition}
\newtheorem{cor}[thm]{Corollary}

\theoremstyle{definition}
\newtheorem{rem}[thm]{Remark}
\newtheorem*{rem*}{Remark}
\newtheorem{defn}[thm]{Definition}

\newtheorem{assump}[thm]{Assumption}

\newtheorem*{notation}{Notation}

\numberwithin{equation}{section}

\def\ds{\displaystyle}
\def\:{\colon}
\def\.{\cdot}
\def\o{\circ}
\def\<{\left\langle}
\def\>{\right\rangle}
\def\({\left(}
\def\){\right)}

\def\epsilon{\varepsilon}
\def\phi{\varphi}

\def\subset{\subseteq}

\def\leq{\leqslant}
\def\geq{\geqslant}
\def\lla{\longleftarrow}

\def\Lra{\Longrightarrow}
\def\IFF{\Longleftrightarrow}

\def\bar#1{\overline{#1}}

\def\tilde#1{\widetilde{#1}}
\def\iso{\cong}

\DeclareMathOperator{\Char}{char}

\DeclareMathOperator{\Id}{Id}

\DeclareMathOperator{\rank}{rank}
\DeclareMathOperator{\coker}{coker}

\def\k{\Bbbk}

\def\Z{\mathbb{Z}}

\DeclareMathOperator{\Coext}{Coext}

\DeclareMathOperator{\End}{End}
\DeclareMathOperator{\Ext}{Ext}

\DeclareMathOperator{\Cohom}{Cohom}
\DeclareMathOperator{\Hom}{Hom}

\DeclareMathOperator{\Map}{Map}

\DeclareMathOperator*{\colim}{colim}

\def\Mod{\mathbf{Mod}}
\def\Stmod{\mathbf{Stmod}}

\DeclareMathOperator{\ann}{ann}
\def\annl{\ann^{\mathrm{l}}}

\def\dlQ{\mathrm{Q}}

\def\QS0{\dlQ S^0}
\def\QSo0{\dlQ_0S^0}

\DeclareMathOperator{\rad}{rad}

\def\op{\mathrm{op}}
\def\sint{{\smallint}}

\DeclareMathOperator{\coind}{coind}
\DeclareMathOperator{\ind}{ind}

\DeclareMathOperator{\cohdim}{coh-dim}
\DeclareMathOperator{\cohrk}{coh-rank}
\def\lint{\sint^{\mathrm{l}}}
\def\rint{\sint^{\mathrm{r}}}
\def\Comod{\mathbf{Comod}}

\def\fd{f.d}
\def\fg{f.g}
\def\fp{f.p}
\def\fr{f.r}
\def\pfr{p.f.r}
\def\lf{l.f}

\title
[Locally Frobenius algebras and Hopf algebras]
{Locally Frobenius algebras and Hopf algebras}
\author{Andrew Baker}
\date{this version \version}
\address{
School of Mathematics \& Statistics,
University of Glasgow, Glasgow G12~8QQ, Scotland.}
\email{andrew.j.baker@glasgow.ac.uk}
\urladdr{http://www.maths.gla.ac.uk/$\sim$ajb}
\thanks{
I would like to thank the following: The
Max-Planck-Institut f\"ur Mathematik in
Bonn for supporting my visit during April
and May 2022; Scott Balchin, Tobias Barthel,
Ken Brown, Bob Bruner, John Rognes and
Chuck Weibel for sharing their mathematical
knowledge and insights. This project really
took off during the first Covid-19 lock-down
in the Spring of 2020, the social isolation
in that strange period at least proved
conducive to mathematical research}
\keywords{Frobenius algebra, Hopf algebra, stable module category}
\subjclass[2020]{Primary 16T05; Secondary 16S99, 57T05}

\begin{document}

\begin{abstract}
We develop a theory of \emph{locally Frobenius
algebras} which are colimits of certain directed
systems of Frobenius algebras. A major goal is
to obtain analogues of the work of Moore \&
Peterson and Margolis on \emph{nearly Frobenius
algebras} and \emph{$P$-algebras} which was
applied to graded Hopf algebras such as the
Steenrod algebra for a prime.

Such locally Frobenius algebras are coherent
and in studying their modules we are naturally
led to focus on coherent and finite dimensional
modules. Indeed, the category of coherent modules
over locally Frobenius algebra~$A$ is abelian
with enough projectives and injectives since~$A$
is injective relative to the coherent modules;
however it only has finite limits and colimits.
The finite dimensional modules also form an
abelian category but finite dimensional modules
are never coherent. The minimal ideals of a
locally Frobenius algebra are precisely the
ones which are isomorphic to coherent simple
modules; in particular it does not contain
a copy of any finite dimensional simple module
so it is not a Kasch algebra. We discuss possible
versions of stable module categories for such
algebras. We also discuss possible monoidal
structures on module categories of a locally
Frobenius Hopf algebra: for example tensor
products of coherent modules turn out to be
pseudo-coherent.

Examples of locally Frobenius Hopf algebras
include group algebras of locally finite
groups, already intensively studied in the
literature.
\end{abstract}

\maketitle

\tableofcontents

\section*{Introduction}

The aim of the paper is to develop a non-graded
version of the theories of nearly \emph{Frobenius
algebras} and \emph{$P$-algebras} introduced
half a century ago by Moore \& Peterson, and
Margolis~\cites{JCM&FPP:NearlyFrobAlg,HRM:Book,AB:Palgebras};
these were motivated by topological applications
involving the Steenrod algebra for a prime number.

In these graded connected versions, the properties
obtained for the algebras and their modules are
reminiscent of properties of Poincar\'e duality
algebras (the graded equivalent of Frobenius
algebras) and we are able to prove similar results.
However there are some difficulties which are
overcome in the graded context by concentrating
on bounded below modules, and it is not clear
how to obtain analogous results in our setting.

Our most complete results involve coherent
modules over \emph{locally Frobenius algebras}
which are themselves coherent rings, although
some results on finite dimensional modules
are also obtained.

One motivation for setting up this theory is
to introduce stable module categories for such
algebras and we discuss options for doing this,
exploiting the fact that a locally Frobenius
algebra is self injective at least relative to
coherent modules and in some cases also relative
to all finite dimensional modules.

Instead of indexing a family of subalgebras
on natural numbers as in the graded theory,
we use a family of augmented Frobenius
subalgebras indexed on a directed set; this
allows us to include examples such as group
algebras of locally finite groups in our
theory. When the indexing is countable this
has implications for vanishing of derived
functors of limits taken over the indexing
set but otherwise we do not make use of its
cardinality.

We give some general examples of locally
Frobenius algebras, but we leave detailed
exploration of examples to future work. The
special case of group algebras of locally
finite groups and its literature was drawn
to my attention by my Glasgow colleague
Ken Brown.

\bigskip
\noindent
\textbf{Notation \& conventions:} Throughout,
$\k$ will denote a field of characteristic~$p\geq0$.
All rings, algebras, modules and their homomorphisms
will be unital.

A \emph{directed set} $(\Lambda,\preccurlyeq)$
will mean a filtered partially ordered set, i.e.,
every pair (or finite set) of elements of $\Lambda$
has an upper bound. A system of objects and morphisms
in a category indexed on such a $(\Lambda,\preccurlyeq)$
will be referred to as a
\emph{$(\Lambda,\preccurlyeq)$-filtered system}
or \emph{$\Lambda$-filtered system}.

\section{Recollections on rings and modules}
\label{sec:Rings-Modules}

We will require results on projective, injective
and flat modules over Noetherian and coherent
rings which are thoroughly covered in
Lam~\cite{TYL:LectModules&Rings}.

The following definitions are standard except
for projectively finitely related.
\begin{defn}\label{defn:fg-fr-fp}
Let $R$ be a ring and $M$ a left/right
$R$-module.
\begin{itemize}
\item
$M$ is \emph{finitely generated} (\fg.)
if there is an exact sequence
\[
F\to M\to 0
\]
with $F$ finitely generated and free (or
equivalently \fg. projective).
\item
$M$ is \emph{finitely related} (\fr.) if
there is a short exact sequence
\[
0\to K\to F\to M\to 0
\]
with $K$ finitely generated and~$F$ free.
\item
$M$ is \emph{projectively finitely related}
(\pfr.) if there is a short exact sequence
\[
0\to K\to P\to M\to 0
\]
with $K$ finitely generated and $P$ projective.
\item
$M$ is \emph{finitely presented} (\fp.)
if there is a short exact sequence
\[
0\to K\to F\to M\to 0
\]
with $K$ \fg. and $F$ \fg. and free (or
equivalently \fg. projective).
\end{itemize}
\end{defn}

The notion of \pfr. is of course redundant
for local rings since every projective module
is free by a celebrated result of
Kaplansky~\cite{IK:ProjMod}.

The next result is a basic tool used in arguments
about such conditions, see for example
Lam~\cite{TYL:LectModules&Rings}*{lemma~(5.1)}.
\begin{lem}[Schanuel's Lemma]\label{lem:Schanuel}
Let $M$ be an $R$-module. Suppose that there
are short exact sequences
\[
0\to K\to P\to M\to 0,\quad 0\to L\to Q\to M\to 0
\]
where~$P$ is projective. Then there is a short
exact sequence of the form
\[
\xymatrix{
0\ar[r] & K\ar[r] & L\oplus P\ar[r] & Q\ar[r] & 0
}
\]
and moreover, if $Q$ is also projective,
\[
P\oplus L \iso K\oplus Q.
\]
\end{lem}
\begin{cor}\label{cor:Schanuel}
Suppose that~$M$ is \fp. and that there is
a short exact sequence
\[
0\to L\to Q\to M\to 0
\]
where $Q$ is \fg., then so is~$L$.
\end{cor}

Using Schanuel's Lemma, it is easy to see
that $\text{(\fg. \& \fr.)}\IFF\text{\fp.}$,
so we use these descriptions interchangeably
as well as using \fg. projective modules in
place of \fg. free modules.

The first two parts of the following result
are taken from
\cite{TYL:LectModules&Rings}*{chapter~2\S4}.
\begin{thm}\label{thm:Lam}
Let $R$ be a ring and $M$ a left/right
$R$-module. \\
\emph{(a)} $M$ is \fr. if and only if
$M \iso M_0 \oplus F$ where~$M_0$ is
\fp. and~$F$ is free. \\
\emph{(b)} If $M$ is \fr. then it is
flat if and only if it is projective. \\
\emph{(c)} $M$ is \pfr. if and only if
$M\oplus F' \iso M_0 \oplus F''$ where~$M_0$
is \fp. and~$F',F''$ are free.
\end{thm}
\begin{proof}
(a) See~\cite{TYL:LectModules&Rings}*{theorem~2.4.26(c)}. \\
(b) See~\cite{TYL:LectModules&Rings}*{theorem~2.4.30}. \\
(c) If $M$ is \pfr. there is a short exact sequence
\[
0\to K\to P\to M\to 0
\]
with $K$ finitely generated and $P$ projective.
Choose a projective module so that $F=Q\oplus P$
is free and write $\ds F^\infty = \bigoplus_{i\in\mathbb{N}}F$.
By the Eilenberg swindle,
\[
P\oplus F^\infty \iso F^\infty
\]
so there is an exact sequence
\[
0\to K\to F^\infty\to M\oplus F^\infty\to 0.
\]
Therefore $M\oplus F^\infty$ is \fr. and
the result follows from~(a).
\end{proof}

Part (c) can be rephrased as saying that
saying that being \pfr. is equivalent to
being \emph{stably \fp.}, or \emph{stably
coherent} when the ring itself is coherent.

We also recall the characterisation of
flat modules provided by the theorem of
Lazard and
Govorov~\cite{TYL:LectModules&Rings}*{theorem~4.34}.
\begin{thm}[Lazard-Govorov theorem]
\label{thm:Lazard-Govorov}
Let $R$ be a ring and $M$ an $R$-module.
Then~$M$ is flat if and only if it is
a filtered colimit of \fg. free modules.
\end{thm}

\subsection*{Coherence for rings and modules}
We recall the notion of (pseudo-)coherence
since we will make heavy use of it. For basic
properties of (pseudo-)coherent modules see
Bourbaki~\cite{Bourbaki:HomAlg}*{X.\S 3, ex.~10}.
\begin{defn}\label{defn:Coherence}
Let $R$ be a ring.
\begin{itemize}
\item
An $R$-module $M$ is \emph{pseudo-coherent}
if every \fg. submodule is \fp..
\item
A \fg. pseudo-coherent module is called
\emph{coherent}.
\item
$R$ is \emph{left/right coherent} if it
is coherent as a left/right $R$-module.
\item
$R$ is \emph{coherent} if it is coherent
as a left and right $R$-module.
\end{itemize}
\end{defn}
\begin{rem}\label{rem:Coherence}
Over a coherent ring, a module is \fp. if
and only if it is coherent, and every \pfr.
module is pseudo-coherent. Moreover, its
coherent modules form an abelian category
(see~Theorem~\ref{thm:fp-AbCat}).

The notion of coherence has an obvious meaning
for algebras over a field, and we will use
it without further comment.
\end{rem}

Of course every Noetherian ring is coherent,
so Frobenius algebras are coherent. Here is
a well known result that can be used to
produce many more examples of coherent rings,
this version appears in
Bourbaki~\cite{Bourbaki:HomAlg}*{X.\S 3, ex.~11e},
and it can be used to show that an infinitely
generated polynomial ring over a coherent
commutative ring is coherent.
\begin{prop}\label{prop:Coherent}
Let $(\Lambda,\preccurlyeq)$ be a directed
set and let $R(\lambda)$ $(\lambda\in\Lambda)$
be a $\Lambda$-directed system of rings
and homomorphisms
$\phi_\alpha^\beta\:R(\alpha)\to R(\beta)$
for $\alpha\preccurlyeq\beta$. Assume that
\begin{itemize}
\item
whenever $\alpha\preccurlyeq\beta$,
$R(\beta)$ is flat as a right/left
$R(\alpha)$-module;
\item
each $R(\lambda)$ is coherent.
\end{itemize}

Then the ring $\ds R=\colim_{(\Lambda,\preccurlyeq)}R(\lambda)$
is left/right coherent and for every
$\lambda\in\Lambda$, $R$ is a flat
right/left $R(\lambda)$-module.
\end{prop}

Noetherianness and coherence of rings are
characterised by homological properties.
\begin{thm}\label{thm:Flat-Inj}
Let $R$ be a ring. \\
\emph{(a)} $R$ is left/right Noetherian
if and only if all coproducts and directed
colimits of left/right injectives are
injective.  \\
\emph{(b)} $R$ is left/right coherent if and
only if all products and directed limits of
flat left/right modules are flat.
\end{thm}
\begin{proof}
(a) See~\cite{TYL:LectModules&Rings}*{theorem~1.3.46}. \\
(b) This is a result of Chase,
see~\cite{TYL:LectModules&Rings}*{theorem~1.4.47}.
\end{proof}

We will frequently make use of faithful flatness,
so for the convenience of the reader we state
some of the main properties.
\begin{prop}\label{prop:FF}
Let $R$ be a ring and $P$ a flat right $R$-module. \\
\emph{(a)} $P$ is faithfully flat if it satisfies
any and hence all of the following equivalent
conditions.
\begin{itemize}
\item
A sequence of left $R$-modules
\[
0\to L\to M\to N\to0
\]
is short exact if and only if the induced
sequence
\[
0\to P\otimes_RL\to P\otimes_RM\to P\otimes_RN\to0
\]
is short exact.
\item
For a left $R$-module $M$, $P\otimes_R M=0$
if and only if $M=0$.
\item
A homomorphism of left $R$-modules $\phi\:M\to N$
is zero if and only if the induced homomorphism
$\Id_P\otimes\phi\:P\otimes_R M\to P\otimes_R N$
is zero.
\end{itemize}
\emph{(b)} If $P$ is faithfully flat then $P\otimes_R(-)$
reflects monomorphisms, epimorphisms and isomorphisms. \\
\emph{(c)} Let $R\to S$ be a ring homomorphism
so that~$S$ is a faithfully flat right $R$-module,
and let~$M$ be a left $R$-module. If $S\otimes_RM$
is a simple $S$-module, then~$M$ is simple.
\end{prop}
\begin{proof}
(a) See~\cite{TYL:LectModules&Rings}*{theorem~4.70}.  \\
(b) This is immediate from (a). \\
(c) By flatness, a short exact sequence of $R$-modules
\[
0\to L\to M\to N\to 0
\]
on tensoring with~$S$ yields a short exact sequence
of $S$-modules
\[
0\to S\otimes_RL\to S\otimes_RM\to S\otimes_RN\to0.
\]
If $S\otimes_RM$ is simple then one of~$S\otimes_RL$
or~$S\otimes_RN$ must be trivial, so by faithful
flatness of~$S$, one of~$L$ or~$N$ must be trivial.
Hence~$M$ is simple.
\end{proof}

%
%
%

\section{Frobenius algebras and Frobenius
extensions}\label{sec:FroAlg&Extns}

Recall that a Frobenius $\k$-algebra~$R$
is self-injective, i.e., injective as a
left/right $R$-module. The following is
a fundamental consequence.

Recall that there are induction and
coinduction functors
\[
\ind_{\k}^{R}\:\Mod_{\k}\to\Mod_{R},
\quad
\coind_{\k}^{R}:\Mod_{\k}\to\Mod_{R}
\]
where
\[
\ind_{\k}^{R}(-)= R\otimes_\k(-),
\quad
\coind_{\k}^{R}(-) = \Hom_\k(R,-).
\]
Notice that for a $\k$-vector space~$W$,
the $R$-module $\ind_{\k}^{R}W$ is
projective (and in fact free) while
$\coind_{\k}^{R}W$ is injective.

For future use we recall that since we
are working over a field $\k$, every
$R$-module~$M$ admits an embedding
into an injective,
\begin{equation}\label{eq:Inj-embedding}
M\to \coind_{\k}^{R}M;
\quad
m\longmapsto (r\mapsto rm).
\end{equation}

The next result is a standard reformulation
of what it means to be a Frobenius algebra.
\begin{thm}\label{thm:Frob-Adjunction}
If $R$ is a Frobenius\/ $\k$-algebra, then
the functors\/ $\ind_{\k}^{R}$ and\/
$\coind_{\k}^{R}$ are naturally isomorphic.
Hence for a\/ $\k$-vector space~$W$,\/
$\ind_{\k}^{R}W\iso\coind_{\k}^{R}W$
is both projective and injective.
\end{thm}

\begin{cor}\label{cor:Frob-Inj=Proj}
An $R$-module is injective if and only
if it is projective.
\end{cor}
\begin{proof}
Let $J$ be an injective $R$-module. As
a special case of~\eqref{eq:Inj-embedding}
there is a monomorphism of $R$-modules
$J\to\coind_{\k}^{R}J$ and by injectivity
there is a commutative diagram with exact
row
\[
\xymatrix{
0\ar[r] & J\ar[d]_{\Id_J}\ar[r] & \coind_{\k}^{R}J\ar[dl] \\
& J &
}
\]
so~$J$ is a retract of a projective module,
hence projective. A similar argument shows
that for a projective~$P$, there is a
commutative diagram
\[
\xymatrix{
\ind_{\k}^{R}P\ar[r] & P\ar[r] & 0 \\
& P\ar[u]_{\Id_P}\ar[ul] &
}
\]
which shows that $P$ is injective.
\end{proof}

A useful consequence of this result is that
every monomorphism of $R$-modules~$R\to M$
splits.

The next result summarises the fundamental
properties of modules over a Frobenius
algebra which motivate much of our work
on modules over locally Frobenius algebras.
\begin{prop}\label{prop:Proj-Flat-Inj}
Let $R$ be a Frobenius\/ $\k$-algebra
and~$M$ an $R$-module. Then the
following are equivalent:
\begin{itemize}
\item
$M$ is injective;
\item
$M$ is projective;
\item
$M$ is flat.
\end{itemize}
\end{prop}
\begin{proof}
The first two are equivalent by
Corollary~\ref{cor:Frob-Inj=Proj}, while
projectivity implies flatness.

Since~$R$ is self-injective and Noetherian,
by using the Lazard-Govorov
Theorem~\ref{thm:Lazard-Govorov},
Corollary~\ref{cor:Frob-Inj=Proj} and
Theorem~\ref{thm:Flat-Inj}(a) we find
that flatness implies injectivity.
\end{proof}

Now recall that if $S$ is a simple $R$-module
then Schur's Lemma implies that then~$\End_R(S)$
is a division algebra central over~$\k$, so $S$ 
can be viewed as a vector space over~$\End_R(S)$
and
\[
\dim_\k S = \dim_\k\End_R(S)\,\dim_{\End_R(S)}S.
\]
\begin{prop}\label{prop:Frob-embeddingInj}
Let $R$ be a Frobenius $\k$-algebra.  \\
\emph{(a)} Every $R$-module embeds into
a free module. In particular, every \fg.
$R$-module embeds into a \fg. free module. \\
\emph{(b)} Let $S$ be a simple left $R$-module.
Then $S$ embeds as a submodule of $R$ with
multiplicity equal to\/ $\dim_{\End_R(S)}S$.
\end{prop}
\begin{proof}
We focus on left modules but the case of 
right modules is similar. \\
(a) For an $R$-module $M$, there is an
injective homomorphism as in~\eqref{eq:Inj-embedding},
\[
M\to\coind_{\k}^{R}M\iso\ind_{\k}^{R}M
\]
and the induced module is a free $R$-module.
When $M$ is \ \fg. $R$-module it is a \fd.
$\k$-vector space so the induced module is
a \fg. free $R$-module.
(b) Every non-trivial homomorphism $S\to R$
is injective. By Schur's Lemma its endomorphism 
algebra is a division algebra which is central 
over~$\k$. The multiplicity of~$S$ in $R$ is 
equal to $\dim_{\End_R(S)^\op}\Hom_R(S,R)$,
where $\End_R(S)$ acts on the right of
$\Hom_R(S,R)$ by precomposition.

The coinduced module
$\coind_{\k}^{R}\k=\Hom_\k(R,\k)$ is a
left $R$-module through the right action
on the domain, and then there are
isomorphisms of right $\End_R(S)$-vector
spaces
\begin{align*}
\Hom_R(S,R)
&\iso \Hom_R(S,\coind_{\k}^{R}\k) \\
&\iso \Hom_\k(R\otimes_RS,\k) \\
&\iso \Hom_\k(S,\k),
\end{align*}
hence $\Hom_R(S,R)$ is non-trivial and
the multiplicity of~$S$ is
\[
\dim_{\End_R(S)^\op}\Hom_R(S,R) = \dim_{\End_R(S)}S.
\qedhere
\]
\end{proof}

An important property of finite dimensional
Hopf algebras is that they have non-zero 
\emph{integrals}. Although in general Frobenius 
algebras do not, Frobenius algebras augmented 
over the ground field do. If $\epsilon\:R\to\k$ 
is the augmentation, its subspaces of left and 
right \emph{integrals with respect to~$\epsilon$}
are defined by
\[
\lint_R = \{ h\in R : \forall y\in R,\;yh=\epsilon(y)h \},
\quad
\rint_R = \{ h\in R : \forall y\in R,\;hy=\epsilon(y)h \}.
\]
Then there is an isomorphism $\lint_R\iso\Hom_R(\k,R)$
and a similar identification of~$\rint_R$ with 
right module homomorphisms.
\begin{prop}\label{prop:FrobAlg-Integrals}
Let $R$ be a Frobenius $\k$-algebra which is
augmented over\/~$\k$. Then viewing~$R$ and\/
$\k$ as left or right $R$-modules we have\/
$\dim_\k\Hom_R(\k,R)=1$ and therefore 
\[
\dim_\k\lint_R=1=\dim_\k\rint_R.
\]
\end{prop}
\begin{proof}
This follows from Proposition~\ref{prop:Frob-embeddingInj}(b).
\end{proof}

If $\rint_R=\lint_R$ then the augmented Frobenius
algebra is \emph{unimodular}; see
Farnsteiner~\cite{RF:FrobExtnHopfAlgs} for more on
this.

We will be interested in Frobenius algebras which
are \emph{symmetric}, i.e., they have a Frobenius
form $\lambda\:A\to\k$ which induces a symmetric
bilinear form. It follows that the associated
Nakayama automorphism is the identity and so
by Farnsteiner~\cite{RF:FrobExtnHopfAlgs}*{lemma~1.1}
it is unimodular.

For a Hopf algebra the counit is the natural
choice of augmentation and we will always choose
it. If the Hopf algebra is commutative or
cocommutative then its antipode is a self-inverse
automorphism making it \emph{involutive}, and
since the identity function is inner,
by~\cite{RF:FrobExtnHopfAlgs}*{proposition~2.3}
it is a symmetric Frobenius algebra if and
only if it is unimodular.

For finite dimensional Hopf algebras there
is an analogue of Maschke's Theorem.
\begin{thm}\label{thm:Maschke-HA}
Let $H$ be a finite dimensional Hopf algebra
and~$\epsilon$ its counit. Then~$H$ is
semisimple if and only if
\[
\epsilon\lint_H\neq\{0\}\neq\epsilon\rint_H.
\]
\end{thm}
\begin{proof}
See~\cite{ML:TourRepThy}*{page~553} for example.
The proof of semisimplicity uses a non-zero
idempotent $e\in\lint_H$ or $e\in\rint_H$ and
the coproduct applied to it.
\end{proof}

\begin{rem}\label{rem:ForbAlg-Macchke}
This result does not apply to all augmented
Frobenius algebras. As an example, consider
$R=\k\times R_0$ with the augmentation being
projection onto the first factor and~$R_0$ a
local Frobenius algebra that is not semisimple.
Let $\lambda_0$ be a Frobenius form on~$R_0$;
then the form on~$R$ given by
$\lambda(x,y)=x+\lambda_0(y)$ is Frobenius
and $(1,0)\in\lint_R$ but~$R$ is not semisimple.
\end{rem}

\bigskip
Now we turn to \emph{Frobenius extensions},
introduced by Kasch~\cite{FK:ProjFrobExtns}.
For a concise account which highlights aspects
relevant to our work, see Lorenz~\cite{ML:TourRepThy},
especially the exercises for sections~2.2
and~12.4. Other useful sources are
Farnsteiner~\cite{RF:FrobExtnHopfAlgs} and
Fischman et al \cite{DF&SM&H-JS:FrobExtns}.
We adopt the notation~$A:B$ rather than~$A/B$
to indicate an extension of algebras.

Suppose given a homomorphism $B\to A$ of
$\k$-algebras so we can view~$A$ as a
$B$-bimodule and consider the extension
of $\k$-algebras~$A:B$.
\begin{defn}\label{defn:FrobExtn}
The extension $A:B$ is a \emph{Frobenius
extension} if there is a $B$-bimodule
homomorphism~$\Lambda\:A\to B$ and 
elements $x_i,y_i\in A$ $(1\leq i\leq n)$ 
such that for all~$a\in A$,
\[
\sum_{1\leq i\leq n}y_i\Lambda(x_ia) 
= a = 
\sum_{1\leq i\leq n}\Lambda(ay_i)x_i.
\]
\end{defn}

\begin{prop}\label{prop:FrobExtn}
Suppose that $A:B$ is a Frobenius extension 
with associated $B$-bimodule homomorphism\/
~$\Lambda\:A\to B$ and elements
$x_i,y_i\in A$ $(1\leq i\leq n)$. Then
$A$ is projective as a left and right
$B$-module. Furthermore, there is an
isomorphism of functors
\[
\coind^A_B\xrightarrow{\;\iso\;}\ind^A_B;
\]
which is defined on each $B$-module~$M$
by
\[
\coind^A_BM=\Hom_B(A,M)\mapsto\coind^A_BM;
\quad
f\mapsto \sum_{1\leq i\leq n}y_if(x_i).
\]
Conversely, if~$A$ is projective as a
left and right $B$-module and\/
$\coind^A_B\iso\ind^A_B$, then~$A:B$
is a Frobenius extension.
\end{prop}

In practice we will be considering the
situation in following result.
\begin{prop}\label{prop:FrobExt-Unimod}
Let $\epsilon\:A\to\k$ be an augmented symmetric
Frobenius algebra and let $B\subseteq A$ be
a subalgebra augmented by the restriction
of $\epsilon$ and which is also a symmetric
Frobenius algebra. Then $A:B$ is a Frobenius
extension.
\end{prop}
\begin{proof}
Since the Nakayama automorphisms of~$A$ and~$B$
are the identity functions, the result is an
immediate consequence
of~\cite{RF:FrobExtnHopfAlgs}*{theorem~1.3}.
\end{proof}

%

\section{Locally Frobenius algebras}\label{sec:LocFrobAlg}

Let $(\Lambda,\preccurlyeq)$ be an infinite
partially ordered set which has the following
properties, in particular it is a \emph{directed
set}:
\begin{itemize}
\item
it is filtered: for any $\lambda',\lambda''\in\Lambda$
there is a common upper bound $\lambda\in\Lambda$,
so $\lambda'\preccurlyeq\lambda$ and
$\lambda''\preccurlyeq\lambda$;
\item
there is a unique initial element $\lambda_0$;
\item
for every $\lambda\in\Lambda$ there is
a $\lambda'''\in\Lambda$ with $\lambda\precneqq\lambda'''$.
\end{itemize}
We will denote such a directed set by
$(\Lambda,\preccurlyeq,\lambda_0)$ but often
just refer to it as~$\Lambda$.

If $\Lambda'\subseteq\Lambda$ is an infinite
filtered subset containing~$\lambda_0$, then
will refer to $(\Lambda',\preccurlyeq,\lambda_0)$
as a \emph{subdirected set} and write
$(\Lambda',\preccurlyeq,\lambda_0)\subseteq(\Lambda,\preccurlyeq,\lambda_0)$.

\begin{defn}\label{defn:LocFrobAlg}
A (necessarily infinite dimensional)
$\k$-algebra~$A$ is \emph{locally
Frobenius of shape\/~$\Lambda$} if it
satisfies the following conditions.  \\
(a) There is a $\Lambda$-directed system
of symmetric Frobenius algebras $A(\lambda)$
($\lambda\in\Lambda$) augmented over~$\k$
and proper inclusion homomorphisms
$\iota_{\lambda'}^{\lambda''}\:
A(\lambda')\hookrightarrow A(\lambda'')$
with $A(\lambda_0)=\k$ and
$\ds A=\bigcup_{\lambda\in\Lambda}A(\lambda)$. \\
(b) Each inclusion $\iota_{\lambda'}^{\lambda''}$
makes $A(\lambda''):A(\lambda')$ a free
Frobenius extension (i.e., $A(\lambda'')$
is free as a left and right $A(\lambda')$-module).
\end{defn}

Of course the Frobenius condition in (b)
is a consequence of Proposition~\ref{prop:FrobExt-Unimod},
while the freeness is an additional assumption.

\begin{defn}\label{defn:LocFrobSubAlg}
Let $A$ be a locally Frobenius algebra
of shape~$\Lambda$ and let
$(\Lambda',\preccurlyeq,\lambda_0)\subseteq(\Lambda,\preccurlyeq,\lambda_0)$
be a subdirected set. A subalgebra~$B\subseteq A$
is a \emph{locally Frobenius subalgebra of shape
$\Lambda'$} if
$\ds B=\bigcup_{\lambda\in\Lambda'}B(\lambda)$ is
locally Frobenius algebra of shape $\Lambda'$ and
for every $\lambda\in\Lambda'$, $B(\lambda)\subseteq A(\lambda)$.
\end{defn}

\begin{notation}
We will denote the kernel of the augmentation
$A(\lambda)\to\k$ by $A(\lambda)^+$; clearly
$A$ is also augmented so we similarly denote
its augmentation ideal by~$A^+$. All of these
are completely prime maximal ideals (recall
that an ideal $P$ in a ring is \emph{completely
prime} if $xy\in P$ implies $x\in P$ or $y\in P$);
in general this is stronger the notion of prime
(an ideal $Q$ in a ring $R$ is \emph{prime} if
$xRy\subseteq P$ implies $x\in P$ or $y\in P$).
All nilpotent elements of~$A$ are contained
in~$A^+$, and if~$e\in A$ is an idempotent
then exactly one of~$e$ or~$1-e$ is in~$A^+$.
\end{notation}

The augmentation condition in (a) may seem
unnecessary but we do require it for some
technical results and it is automatically
satisfied in the Hopf algebra case because
the counit is an augmentation. The Frobenius
extension condition in (b) might be weakened,
and even in the Hopf algebra case it is not
otherwise guaranteed except in special
circumstances such as when the~$A(\lambda)$
are all local.

We make a trivial observation.
\begin{lem}\label{lem:finitesubset}
Let $A$ be a locally Frobenius algebra.
Then for every finite subset $Z\subseteq A$
there is a $\lambda\in\Lambda$ such that
$Z\subseteq A(\lambda)$. Hence every
finite dimensional subspace~$V\subseteq A$
is contained in some~$A(\lambda)$.
\end{lem}
\begin{proof}
Each element $z\in Z$ is contained in
some $A(\lambda_z)$ and by the filtering
condition there is an upper bound $\lambda$
of the $\lambda_z$, so that
$Z\subseteq A(\lambda)$.
\end{proof}

Of course a \fd. subalgebra of $A$ is
\fd. subspace so it is a subalgebra
of some~$A(\lambda)$.

The next result provides an important
class of examples of locally Frobenius
algebras which we will refer to as
\emph{locally Frobenius Hopf algebras}.

\begin{prop}\label{prop:LOcFrobHopfAlg}
Let $H$ be a Hopf algebra over\/ $\k$.
Then~$H$ is a locally Frobenius algebra
of shape~$\Lambda$ if the following
conditions are satisfied.    \\
\emph{(a)} There is a $\Lambda$-directed
system of \fd. subHopf algebras $H(\lambda)\subseteq H$
$(\lambda\in\Lambda)$ and inclusion homomorphisms
$\iota_{\lambda'}^{\lambda''}\:
H(\lambda')\to H(\lambda'')$ with $H(\lambda_0)=\k$
and
$\ds H=\bigcup_{\lambda\in\Lambda}H(\lambda)$.  \\
\emph{(b)} Each $H(\lambda)$ is involutive
and unimodular.
\end{prop}
\begin{proof}
The Larson-Sweedler theorem~\cite{LarsonSweedlerThm}
implies that each $H(\lambda)$ is Frobenius,
while~\cite{RF:FrobExtnHopfAlgs}*{proposition~2.3}
implies that it is symmetric since the square of
the antipode is the identity which is also a Nakayama
automorphism. Finally, the
Nichols-Zoeller theorem~\cite{WDN&MBZ:HAfreeness}.
implies that each pair $H(\lambda''):H(\lambda')$
is free.
\end{proof}

Of course if each Hopf algebra $H(\lambda)$ is
commutative or cocommutative then it is involutive.

It is not necessarily true that an infinite
dimensional Hopf algebra~$H$ is free as a
left/right module over a finite dimensional
subHopf algebra~$K$, although some results
on when this holds are known; for the case
where~$K$ is semisimple see~\cite{WDN&MBR:HAfreenessinfdim},
while for the case where~$K$ is a normal
subalgebra see~\cite{H-JS:RemQGps}. However,
for our purposes this is not required.

Later we will discuss some examples of locally
Frobenius Hopf algebras such as group algebras
of locally finite groups.

The next definition builds on
Definition~\ref{defn:LocFrobSubAlg}.
\begin{defn}\label{defn:LocFrobSubHopfAlg}
Let $H$ be a locally Frobenius Hopf algebra of
shape $\Lambda$ and let
$(\Lambda',\preccurlyeq,\lambda_0)\subseteq(\Lambda,\preccurlyeq,\lambda_0)$
be a subdirected set. A subHopf algebra~$K\subseteq H$
is a \emph{locally Frobenius subHopf algebra
of shape $\Lambda'$} if $\ds K=\bigcup_{\lambda\in\Lambda'}K(\lambda)$ is
locally Frobenius Hopf algebra of shape $\Lambda'$
and for every $\lambda\in\Lambda'$,
$K(\lambda)\subseteq H(\lambda)$.
\end{defn}

We end this section with a result suggested
by Richardson~\cite{JSR:GpRngsnonzerosocle}*{lemma~2.2}.
This shows that the augmentation ideal of a
locally Frobenius subalgebra generates a
`large' submodule.
\begin{prop}\label{prop:JSR-lemma2.2}
Let $A$ be a locally Frobenius algebra of shape
$\Lambda$ and $B$ a locally Frobenius subalgebra
of shape $\Lambda'\subseteq\Lambda$. Then the
submodule $AB^+\subseteq A$ is essential.
\end{prop}
\begin{proof}
This proof is a straightfoward adaption of that
in Richardson~\cite{JSR:GpRngsnonzerosocle}.

Suppose that $Az\cap AB^+=0$ for some non-zero
$z\in A$. Choose $\alpha\in\Lambda$ such that
$z\in A(\alpha)$. As $\Lambda'$ is infinite
we can choose a $\beta\in\Lambda'$ for which
\[
\dim_\k A(\beta) >
\frac{\dim_\k A(\beta)}{\dim_\k A(\beta)z}.
\]
Now choose $\gamma\in\Lambda'$ so that
\[
A(\alpha)\subseteq A(\gamma)\supseteq A(\beta).
\]
and then
\[
A(\gamma)z\cap A(\gamma)A(\beta)^+
\subseteq Az\cap A(\gamma)B^+ = 0.
\]
Now recall that $A(\gamma)$ is a free module
over each of $A(\alpha)$ and $A(\beta)$, so
\[
A(\gamma)z \iso A(\gamma)\otimes_{A(\alpha)}A(\alpha)z,
\quad
A(\gamma)A(\beta)^+ \iso A(\gamma)\otimes_{A(\beta)}A(\beta)^+,
\]
giving
\begin{align*}
\dim_\k A(\gamma)z &=
\frac{\dim_\k A(\gamma)}{\dim_\k A(\alpha)}\dim_\k A(\alpha)z, \\
\dim_\k A(\gamma)A(\beta)^+ &=
\frac{\dim_\k A(\gamma)}{\dim_\k A(\beta)}\bigl(\dim_\k A(\beta)-1\bigr).
\end{align*}
Using these we obtain
\begin{align*}
\dim_\k A(\gamma) &\geq
\dim_\k\bigl(A(\gamma)z\oplus A(\gamma)A(\beta)^+\bigr) \\
&= \dim_\k A(\gamma)z + \dim_\k A(\gamma)A(\beta)^+ \\
&=
\frac{\dim_\k A(\gamma)\dim_\k A(\alpha)z}{\dim_\k A(\alpha)}
+
\frac{\dim_\k A(\gamma)\bigl(\dim_\k A(\beta)-1\bigr)}{\dim_\k A(\beta)} \\
&>
\frac{\dim_\k A(\gamma)}{\dim_\k A(\beta)}
+
\frac{\dim_\k A(\gamma)\bigl(\dim_\k A(\beta)-1\bigr)}{\dim_\k A(\beta)} \\
&= \dim_\k A(\gamma),
\end{align*}
which is impossible. Therefore $AB^+\subseteq A$
must be an essential submodule.
\end{proof}

\section{Modules over a locally Frobenius
algebra}\label{sec:LocFrob-CohMods}

Now we will describe some basic properties
of modules over locally Frobenius algebras.
Throughout this section we will suppose that~$A$
is a locally Frobenius of shape~$\Lambda$.

First we state a result for coherent $A$-modules;
since~$A$ is coherent these are exactly the \fp.
$A$-modules.
\begin{thm}\label{thm:fp-AbCat}
The category\/ $\Mod^{\mathrm{coh}}_A$ of coherent
$A$-modules and their homomorphisms is an abelian
category with all finite limits and colimits as
well as enough projectives and injectives.
\end{thm}
\begin{proof}
Clearly this category has finite products and
coproducts, so it is sufficient to check that
it has kernels, images and cokernels; this is
an exercise in
Bourbaki~\cite{Bourbaki:HomAlg}*{ex.~\S3.10(b)},
see also Cohen~\cite{JMC:Coherent}*{section~1}.
For projectives and injectives see
Lemma~\ref{lem:Proj&Inj} below.
\end{proof}

\begin{prop}\label{prop:set-up}{\ } \\
\emph{(a)} For each $\lambda\in\Lambda$,~$A$
is injective, projective and flat as a left
or right $A(\lambda)$-module. \\
\emph{(b)} The $\k$-algebra $A$ is coherent. \\
\emph{(c)} Suppose that~$M$ is a coherent $A$-module.
Then for some~$\lambda$ there is an $A(\lambda)$-module~$M'$
with a finite presentation
\[
0 \lla M' \lla A(\lambda)^k \lla A(\lambda)^\ell
\]
inducing a finite presentation
\[
0 \lla A\otimes_{A(\lambda)}M' \lla A^k\lla A^\ell
\]
where $A\otimes_{A(\lambda)}M'\iso M$.  \\
\emph{(d)}
Let $\phi\:M\to N$ be a homomorphisms of coherent
$A$-modules. Then there is a $\lambda\in\Lambda$,
and a homomorphism of finitely generated
$A(\lambda)$-modules $\phi''\:M''\to N''$ fitting
into a commutative diagram of $A$-module
homomorphisms.
\[
\xymatrix{
A\otimes_{A(\lambda)}M''\ar[rr]^{\Id\otimes\phi''}\ar@{<->}[d]_\iso
&& A\otimes_{A(\lambda)}N''\ar@{<->}[d]^\iso \\
M\ar[rr]^\phi && N
}
\]
\end{prop}
\begin{proof}
We will give brief indications of the proofs.  \\
(a) By Proposition~\ref{prop:Coherent}
$A$ is a flat $A(\lambda)$-module, hence
by Proposition~\ref{prop:Proj-Flat-Inj}
it is also injective and projective. \\
(b) Since each $A(\lambda)$ is Noetherian
over~$\k$, Proposition~\ref{prop:Coherent}
implies that~$A$ is left and right coherent.  \\
(c) If $M$ is an \fp. $A$-module, there
is an exact sequence of $A$-modules
\[
0 \lla M \xleftarrow{\rho_0} A^m \xleftarrow{\rho_1}A^n
\]
for some $m,n$. The image of $\rho_1$ must
be contained in $A(\lambda)^m\subseteq A^m$
for some $\lambda\in\Lambda$, so we obtain
an exact sequence of $A(\lambda)$-modules
\[
0 \lla M' \xleftarrow{\rho'_0}A(\lambda)^m
\xleftarrow{\rho'_1} A(\lambda)^n
\]
and on tensoring with $A$, by (a) this yields
an exact sequence of $A$-modules
\[
0 \lla A\otimes_{A(\lambda)}M'
\xleftarrow{\Id\otimes\rho'_0} A\otimes_{A(\lambda)}A(\lambda)^m
\xleftarrow{\Id\otimes\rho'_1} A\otimes_{A(\lambda)}A(\lambda)^n
\]
which is equivalent to the original one.     \\
(d) Using (c) we can represent $M$ and $N$
as induced up from \fg. $A(\lambda_1)$-modules
$M',N'$ for some~$\lambda_1$. The image of the
restriction of~$\phi$ to~$M'$ in $A\otimes_{A(\lambda_1)}N'$
lies in some $A(\lambda_2)\otimes_{A(\lambda_1)}N'$
where $\lambda_1\preccurlyeq\lambda_2$. Base
changing gives a homomorphism
\[
\phi''\:M''=A(\lambda_2)\otimes_{A(\lambda_1)}M'
\to A(\lambda_2)\otimes_{A(\lambda_1)}N'=N''
\]
with the required properties.
\end{proof}

The faithful flatness condition of Proposition~\ref{prop:A-Faithfullyflat}
gives another useful property.
\begin{cor}\label{cor:set-up-SES}
Every short exact sequence of coherent $A$-modules
\begin{equation}\label{eq:set-up-SES}
0\to L\xrightarrow{\phi} M\xrightarrow{\theta}N\to0
\end{equation}
is induced up from a short exact sequence of \fg.
$A(\alpha)$-modules
\[
0\to L'\xrightarrow{\phi'} M'\xrightarrow{\theta'}N'\to0
\]
for some $\alpha\in\Lambda$, i.e., there is
a commutative diagram of $A$-modules of the
following form.
\[
\xymatrix{
0\ar[r] & L\ar[r]^\phi\ar@{<->}[d]_\iso
& M\ar[r]^\theta\ar@{<->}[d]_\iso
& N\ar[r]\ar@{<->}[d]_\iso
& 0 \\
0\ar[r] & A\otimes_{A(\alpha)}L'\ar[r]^{\Id\otimes\phi'}
& A\otimes_{A(\alpha)}M'\ar[r]^{\Id\otimes\theta'}
& A\otimes_{A(\alpha)}N'\ar[r]
& 0
}
\]
In particular, an isomorphism $\phi\:L\to M$
is induced up from an isomorphism of \fg.
$A(\alpha)$-modules for some $\alpha\in\Lambda$.
\end{cor}
\begin{proof}
Choose $\alpha\in\Lambda$ so that there are
homomorphisms of \fg. $A(\alpha)$-modules
\[
L'\xrightarrow{\phi'}M'\xrightarrow{\theta'}N'
\]
such that
\[
0\to
A\otimes_{A(\alpha)}L'\xrightarrow{\Id\otimes\phi'}A\otimes_{A(\alpha)}M'
\xrightarrow{\Id\otimes\theta'}A\otimes_{A(\alpha)}N'
\to0
\]
corresponds to the short exact sequence~\eqref{eq:set-up-SES}.
Then by Proposition~\ref{prop:FF}(a),
\[
0\to L'\xrightarrow{\phi'} M'\xrightarrow{\theta'}N'\to0
\]
is a short exact sequence of $A(\alpha)$-modules.

The statement about isomorphisms also follows using
faithful flatness of the $A(\alpha)$-module~$A$ and
Proposition~\ref{prop:FF}(b).
\end{proof}

Now we give some results on the dimension of
coherent $A$-modules.
\begin{lem}\label{lem:A/ALfindim}
Let $\lambda\in\Lambda$ and let $L\subseteq A(\lambda)$
be a left ideal. Then as left $A$-modules,
\[
A\otimes_{A(\lambda)}A(\lambda)/L \iso A/AL.
\]
and the $\k$-vector space $A/AL$ is infinite dimensional.
\end{lem}
\begin{proof}
By a standard criterion for flatness
of~\cite{TYL:LectModules&Rings}*{(4.12)}, multiplication
induces an isomorphism
\[
A\otimes_{A(\lambda)}L\xrightarrow{\iso}AL
\]
where $AL\subseteq A$ is the left ideal generated
by~$L$. Therefore there is a commutative diagram
of left $A$-modules
\[
\xymatrix{
0\ar[r] &  A\otimes_{A(\lambda)}L\ar[r]\ar@{<->}[d]_\iso
& A\otimes_{A(\lambda)}A(\lambda)\ar[r]\ar@{<->}[d]_\iso
& A\otimes_{A(\lambda)}A(\lambda)/L\ar[r]\ar[d] & 0  \\
0\ar[r] &  AL\ar[r] & A\ar[r] & A/AL\ar[r] & 0
}
\]
with exact rows. It follows that the right hand
vertical arrow is an isomorphism.

Whenever $\alpha\preccurlyeq\beta$, $A(\beta)$
is a finite rank free $A(\alpha)$-module, therefore
\[
\dim_\k A(\beta) = \dim_\k A(\alpha)\rank_{A(\alpha)} A(\beta)
\]
and $\rank_{A(\alpha)} A(\beta)\leq\rank_{A(\alpha)}A$
where is strictly increasing as a function of~$\beta$.
\end{proof}

\begin{prop}\label{prop:cohmod-infdim}
Let $M$ be a coherent left/right $A$-module. Then~$M$
is an infinite dimensional\/ $\k$-vector space.
%
%
\end{prop}
\begin{proof}
We will assume that $M$ is a left module,
the proof when it is a right module is
similar.

When $M$ is cyclic, $M\iso A/L$ for some
\fg. left ideal~$L$. A finite set of
generators of $L$ must lie in some~$A(\lambda)$
so the ideal $L(\lambda)=A(\lambda)\cap L$
satisfies $AL(\lambda) = L$ and by
Lemma~\ref{lem:A/ALfindim}, $M\iso AL(\lambda)$
is infinite dimensional.

Now for a general coherent module, suppose
that~$M$ has~$m$ generators $x_1,\dots,x_m$
where~$m$ is minimal. Then for the proper
submodule $M'\subset M$ generated by
$x_1,\dots,x_{m-1}$, $M/M'$ is a non-trivial
cyclic coherent module which is infinite
dimensional. Since there is an epimorphism
$M\to M/M'$, $M$ must be infinite dimensional.
%
%
\end{proof}

This gives an important fact about \fd. modules.
%
%
%
\begin{prop}\label{prop:fd-noncoherent}
A non-trivial \fd. $A$-module is not coherent.
\end{prop}

For simple modules we have the following.
\begin{cor}\label{cor:fd-simple-noncoherent}
Let $S=A/L$ be a \fd. simple left/right $A$-module,
where~$L$ is a maximal left/right ideal of~$A$.
Then~$S$ is not coherent and~$L$ is not finitely
generated. In particular, the trivial $A$-module\/~$\k$
is not coherent and the augmentation ideal~$A^+$
is not \fg. as a left/right module.
\end{cor}

Of course this shows that a locally Frobenius
algebra is never Noetherian. As we will see
later, it also implies that the trivial
module is not isomorphic to a submodule
of~$A$.
\begin{cor}\label{cor:A/fgideal-nilpotent}
Let $z\in A$ be nilpotent. Then $A/Az$
and $A/zA$ are both infinite dimensional.
\end{cor}
\begin{proof}
Every nilpotent element of~$A$ is contained
in the completely prime ideal~$A^+$.
\end{proof}

We end this section with another important
observation.
\begin{prop}\label{prop:A-Faithfullyflat}
For $\alpha,\beta\in\Lambda$,
\begin{itemize}
\item
if $\alpha\preccurlyeq\beta$ then $A(\beta)$
is a faithfully flat $A(\alpha)$-module;
\item
$A$ is a faithfully flat $A(\alpha)$-module.
\end{itemize}
\end{prop}
\begin{proof}
We know that $A(\beta)$ and $A$ are flat as
$A(\alpha)$-modules; we also know that the
inclusions $A(\alpha)\hookrightarrow A(\beta)$
and $A(\alpha)\hookrightarrow A$ split as
$A(\alpha)$-module homomorphisms since
$A(\alpha)$ is self-injective.

For an $A(\alpha)$-module $M$, the unit induces
a split homomorphism of $A(\alpha)$-modules
$M\to A(\beta)\otimes_{A(\alpha)}M$, showing
that $A(\beta)$ is a faithful $A(\alpha)$-module.

\[
\xymatrix{
M\ar[r]_(.3)\iso\ar@/^15pt/[rr]\ar@/_23pt/[drr]_\iso & A(\alpha)\otimes_{A(\alpha)}M\ar[r]
& A(\beta)\otimes_{A(\alpha)}M\ar@{-->}[d]   \\
&& A(\alpha)\otimes_{A(\alpha)}M
}
\]
A similar argument with $A$ in place of~$A(\beta)$
shows that~$A$ is also faithful.
\end{proof}

Combining Lemma~\ref{prop:FF}(c) with
Proposition~\ref{prop:A-Faithfullyflat} we
obtain
\begin{cor}\label{cor:A-reflectssimple}
Let $\alpha,\beta\in\Lambda$ and let $M$ be
a left $A(\alpha)$-module. Then
\begin{itemize}
\item
for $\alpha\preccurlyeq\beta$, if
$A(\beta)\otimes_{A(\alpha)}M$ is a simple
$A(\beta)$-module then $M$ is simple;
\item
if $A\otimes_{A(\alpha)}M$ is a simple
$A$-module then~$M$ is simple.
\end{itemize}
\end{cor}

Here are some useful consequences of faithful
flatness of~$A$.
\begin{cor}\label{cor:SimpleCoherent}
Let $M$ be a simple coherent $A$-module. Then
there is a~$\lambda\in\Lambda$ for which~$M$
has the form $M\iso A\otimes_{A(\lambda)}M'$
for a simple~$A(\lambda)$-module~$M'$.
\end{cor}
\begin{proof}
By Proposition~\ref{prop:set-up} there is a
$\lambda\in\Lambda$ such that~$M$ is induced
up from an $A(\lambda)$-module~$M'$. By
Corollary~\ref{cor:A-reflectssimple},~$M'$
is simple.
\end{proof}
\begin{cor}\label{cor:A-reflectsepimono}
Let $\lambda\in\Lambda$, and let\/
$\phi\:M_\lambda\to N_\lambda$ be a homomorphism
of \fg. left $A(\lambda)$-modules. If\/
$\Id\otimes\phi\:A\otimes_{A(\lambda)}M_\lambda
\to A\otimes_{A(\lambda)}N_\lambda$
is a monomorphism/an epimorphism/an isomorphism,
then so is\/~$\phi$.
\end{cor}
\begin{proof}
See Proposition~\ref{prop:FF}(b).
\end{proof}

The Frobenius extension condition for each
inclusion $A(\lambda)\subseteq A(\lambda')$
ensures that the induction and coinduction
functors
$\ind_{A(\lambda)}^{A(\lambda')}\:\Mod_{A(\lambda)}\to\Mod_{A(\lambda')}$
and
$\coind_{A(\lambda)}^{A(\lambda')}:\Mod_{A(\lambda)}\to\Mod_{A(\lambda')}$
are naturally isomorphic, where for a left
$A(\lambda)$-module~$M$,
\[
\ind_{A(\lambda)}^{A(\lambda')}M
= A(\lambda')\otimes_{A(\lambda)}M,
\quad
\coind_{A(\lambda)}^{A(\lambda')}M
= \Hom_{A(\lambda)}(A(\lambda'),M).
\]
Here we use the right multiplication of $A(\lambda')$
on itself to define the left $A(\lambda')$-module
structure on $\Hom_{A(\lambda)}(A(\lambda'),M)$.
Of course we can specialise to the case
$\lambda=\lambda_0$ and $A(\lambda_0)=\k$.

For an $A(\lambda)$-module~$M$, there is
injective composition
\[
M \xrightarrow{\iso} \Hom_\k(\k,M)
\to \Hom_\k(A(\lambda),M);
\quad
x\mapsto (a\mapsto ax)
\]
which is an $A(\lambda)$-module homomorphism
with
\[
\Hom_\k(A(\lambda),M)
\iso \coind_{\k}^{A(\lambda)}M
\iso\ind_{\k}^{A(\lambda)}M
\]
being both an injective $A(\lambda)$-module
and a free module $A(\lambda)$-module. Of
course if~$M$ is \fg. then so is
$\ind_{\k}^{A(\lambda)}M$.

\begin{prop}\label{prop:fp-injective}
Let $M$ be a coherent $A$-module. Then
there is an embedding of~$M$ into a \fg.
free $A$-module.
\end{prop}
\begin{proof}
We know that $M\iso\ind_{A(\lambda)}^AM'$
for some \fg. $A(\lambda)$-module with
$\lambda\in\Lambda$. There is also an
embedding of~$M'$ into a \fg. free
$A(\lambda)$-module~$F'$ say. Inducing
up and using flatness of~$A$ over
$A(\lambda)$, we obtain an injection
\[
M\xrightarrow{\iso}\ind_{A(\lambda)}^AM'
         \to\ind_{A(\lambda)}^AF'=F
\]
where $F$ is a \fg. free $A$-module.
\end{proof}

\begin{lem}\label{lem:Ainj-fp}
If $M$ is a coherent $A$-module, then
for $s>0$, $\Ext_A^s(M,A)=0$. Hence~$A$
is injective in the category of coherent
$A$-modules. More generally, this holds
if~$M$ is a coproduct of coherent
$A$-modules.
\end{lem}
\begin{proof}
By Proposition~\ref{prop:set-up}(c), for
some $\lambda\in\Lambda$ there is an
$A(\lambda)$-module $M'$ such that
$A\otimes_{A(\lambda)}M'\iso M$. Then
\[
\Ext^*_A(M,A)
\iso \Ext^*_A(A\otimes_{A(\lambda)}M',A)
\iso \Ext^*_{A(\lambda)}(M',A).
\]
Now recall that by Proposition~\ref{prop:set-up}(a),
for $\lambda\in\Lambda$, $A$ is an injective
$A(\lambda)$-module, hence for~$s>0$,
$\Ext^s_{A(\lambda)}(M',A)=0$.

An alternative argument uses Proposition~\ref{prop:A-Faithfullyflat}.
A sequence of coherent $A$-modules
\[
0\to U\to V
\]
is induced up from a sequence of $A(\lambda)$-modules
for some $\Lambda\in\Lambda$,
\[
0\to A\otimes_{A(\lambda)}U'\to A\otimes_{A(\lambda)}V'
\]
which is exact if and only if the sequence
\[
0\to U'\to V'
\]
is exact. Now given the solid diagram of $A$-modules
\[
\xymatrix{
& A\otimes_{A(\lambda)}U'\ar@{<->}[d]^\iso & A\otimes_{A(\lambda)}V'\ar@{<->}[d]^\iso \\
0\ar[r] & U\ar[r]\ar[d] & V\ar@{-->}[ld] \\
& A &
}
\]
with exact row, by using the adjunction there
is a solid diagram of $A(\lambda)$-modules
\[
\xymatrix{
0\ar[r] & U'\ar[r]\ar[d] & V'\ar@{-->}[ld] \\
& A &
}
\]
with exact row. Since $A$ is a flat $A(\lambda)$-module,
it is injective and it follows that we can complete this
diagram with a dashed arrow. Again using the adjunction,
we can complete the original diagram, showing that~$A$
is relatively injective.

For a coproduct of coherent modules, it is
standard that $\Ext_A^s(-,N)$ sends coproducts
to products.
\end{proof}

\begin{cor}\label{cor:Ainj-fr}
If $M$ is a \fr. $A$-module, then for $s>0$,
$\Ext_A^s(M,A)=0$.
\end{cor}
\begin{proof}
By Theorem~\ref{thm:Lam}(a), such a module
has the form $M\iso M_0\oplus F$, with~$M_0$
\fp. and~$F$ free. Then for~$s>0$,
\[
\Ext_A^s(M,A) \iso \Ext_A^s(M_0,A) = 0.
\qedhere
\]
\end{proof}

\begin{prop}\label{prop:fr-embedding}
Let $A$ be a locally Frobenius $\k$-algebra
and $M$ an $A$-module. \\
\emph{(a)} If $M$ is \fp. then there is
an embedding of $M$ into a finite rank
free module. \\
\emph{(b)} If $M$ is \fr. then there is
an embedding of~$M$ into a free module.
\end{prop}
\begin{proof}
(a) As in the proof of Lemma~\ref{lem:Ainj-fp},
a finite presentation
\[
A^m\to A^n\to M\to 0
\]
is induced up from an exact sequence of
$A(\lambda)$-modules
\[
A(\lambda)^m\to A(\lambda)^n\to M'\to 0
\]
for some $\lambda\in\Lambda$, where
$A\otimes_{A(\lambda)}M'\iso M$. By
Proposition~\ref{prop:Frob-embeddingInj}(a),
there is a monomorphism $M'\to F'$ where~$F'$
is a \fg. free $A(\lambda)$-module. By
flatness of $A$ over $A(\lambda)$, this
gives a monomorphism
\[
M \xrightarrow{\iso}A\otimes_{A(\lambda)}M'
\to A\otimes_{A(\lambda)}F' = F,
\]
where $F$ is a \fg. free $A$-module.  \\
(b) By Theorem~\ref{thm:Lam}(a), $M\iso F\oplus M_0$
with $F$ free and $M_0$ \fp., so the result
follows using~(a).
\end{proof}

We now have a result which completes the proof
of Theorem~\ref{thm:fp-AbCat}.
\begin{lem}\label{lem:Proj&Inj}
The abelian category $\Mod^{\mathrm{coh}}_A$
has enough projectives and injectives which
are the summands of \fg. free modules.
\end{lem}
\begin{proof}
The existence of projectives is obvious.
Lemma~\ref{lem:Ainj-fp} implies the existence
of injectives in~$\Mod^{\mathrm{coh}}_A$, and
by Proposition~\ref{prop:fr-embedding}(a),
every coherent module embeds in a \fg. free
module which is also injective.
\end{proof}

\begin{rem}\label{rem:Coherent-Ext}
Of course we should not expect~$A$ to be an
injective in~$\Mod_A$ so it is only relatively
injective with respect to the full
subcategory~$\Mod^{\mathrm{coh}}_A$; the basic
notions of relative homological algebra can be
found in Eilenberg \& Moore~\cite{SE&JCM:RelHomAlg}.
We can use finitely generated free modules to
build projective resolutions in~$\Mod^{\mathrm{coh}}_A$
for computing~$\Ext_A$ since for a coherent
module~$M$,~$\Ext_A^*(-,M)$
are isomorphic to derived functors
on~$\Mod^{\mathrm{coh}}_A$. On the category of all
$A$-modules the left exact functor $\Hom_A(M,-)$
has the right derived functors $\Ext_A^*(M,-)$.
By Lemma~\ref{lem:Ainj-fp}, for each finitely
generated free module~$F$, $\Ext_A^s(M,F) = 0$
if $s>0$. This means that~$F$ is $\Hom_A(M,-)$-acyclic,
and it is well-known that these right derived
functors can be computed using resolutions by
such modules which always exist here; see~\cite{CAW:HomAlg}
on $F$-acyclic objects and dimension shifting.
\end{rem}

\subsection*{Injective, projective and flat modules}
Lemma~\ref{lem:Ainj-fp} shows that~$A$ is
relatively injective with respect to the
category of coherent $A$-modules; it follows
that \fg. free and projective modules are
also relatively injective. We can relate
flatness to relatively injectivity.
\begin{lem}\label{lem:flat->relinj}
Let $P$ be a flat $A$-module. Then $P$ is
relatively injective with respect to the
category of coherent $A$-modules.
\end{lem}
\begin{proof}
By~\cite{TYL:LectModules&Rings}*{theorem~4.32},
every homomorphism $M\to P$ from a finitely
presented $A$-module factors as
\[
\xymatrix{
M\ar[r]\ar@/^15pt/[rr] & F\ar[r] & P
}
\]
where $F$ is a finitely generated free module.
Then given a diagram of solid arrows with exact
row where~$U,V$ are coherent,
\[
\xymatrix{
0\ar[r] & U\ar[rr]\ar[dd]\ar@{-->}[dr] && V\ar@{.>}[dl]  \\
& & F\ar@{-->}[dl] & \\
 & P &&
}
\]
we can extend it with the dashed arrows using
the factorisation above, and since~$F$ is
relatively injective there is a dotted arrow
making the whole diagram commute.

This can also be proved using the Lazard-Govorov
Theorem~\ref{thm:Lazard-Govorov}.
\end{proof}

Since products of (relative) injectives are
(relative) injectives, this is consistent
with Chase's Theorem~\ref{thm:Flat-Inj}(b)
which says that products of flat modules
over a coherent ring are flat.

\begin{lem}\label{lem:inj->flat}
Let $J$ be a relatively injective $A$-module
with respect to the category of coherent
$A$-modules. Then~$J$ is flat. In particular,
every injective $A$-module is flat.
\end{lem}
\begin{proof}
Let $M$ be a coherent $A$-module. By
Proposition~\ref{prop:fp-injective} there
is a monomorphism~$i\:M\to F$ where~$F$
is a \fg. free module. Now for any homomorphism
$f\:M\to J$, the diagram of solid arrows with
exact row
\[
\xymatrix{
0\ar[r] & M\ar[r]^i\ar[d]_f & F\ar@{-->}[dl] \\
& J &
}
\]
can be extended by a dashed arrow by injectivity.
By
Lam~\cite{TYL:LectModules&Rings}*{theorem~4.32},
$J$ is flat.
\end{proof}

Here is a summary of what we have established.
\begin{prop}\label{prop:PalgebrasProj-Flat-Inj}
Let~$M$ be an $A$-module. Consider the following
conditions:
\begin{itemize}
\item[(A)]
$M$ is relatively injective for the category
of coherent $A$-modules;
\item[(B)]
$M$ is flat;
\item[(C)]
$M$ is projective;
\item[(D)]
$M$ is a product of flat modules.
\end{itemize}
Then we have the implications shown.
\[
\xymatrix{
\text{\rm(A)}\ar@{<=>}[r] & \text{\rm(B)} & \ar@{=>}[l]\text{\rm(C)}  \\
& \text{\rm(D)}\ar@{<=>}[u] &
}
\]
Furthermore, if $M$ is coherent or more generally
\pfr., then $\text{\rm(B)}\Lra\text{\rm(C)}$.
\end{prop}
\begin{proof}
The equivalence of (A) and (B) follows from
Lemmas~\ref{lem:flat->relinj} and~\ref{lem:inj->flat}.
The implication $\text{(D)}\Lra\text{(B)}$
follows from Chase's
theorem~\cite{TYL:LectModules&Rings}*{theorem~4.47};
in particular it applies to the case of a
product of free modules. Actually we don't
really need to use Chase's Theorem since
a product of (relative) injectives is
a (relative) injective.

The last statement follows by Theorem~\ref{thm:Lam}(a)
and Lam~\cite{TYL:LectModules&Rings}*{theorem~(4.30)}.
\end{proof}

We also mention another result on flat modules
whose proof requires
Corollary~\ref{cor:fd->Coherent} which will
be given later.
\begin{prop}\label{prop:Ext(fin,flat)}
Let $M$ be a \fd. $A$-module and~$P$ a flat
$A$-module. Then
\[
\Hom_A(M,P) = 0.
\]
\end{prop}
\begin{proof}
By the Lazard-Govorov Theorem~\ref{thm:Lazard-Govorov},
$P$ is a filtered colimit of \fg. free modules~$F_\alpha$,
\[
P = \colim_\alpha F_\alpha.
\]
Since $M$ is \fg. and by Corollary~\ref{cor:fd->Coherent},
\[
\Hom_A(M,P) = \colim_\alpha\Hom_A(M,F_\alpha) = 0.
\qedhere
\]
\end{proof}

\subsection*{Jacobson Radicals}
Recall for $\alpha\preccurlyeq\beta$, $A(\alpha)$
is left/right self-injective and so each inclusion
$A(\alpha)\hookrightarrow A(\beta)$ is split
as a left/right $A(\alpha)$-module homomorphism.
So by Lam~\cite{TYL:NonCommRings}*{proposition~5.6},
\begin{equation}\label{eq:rad-alpha->beta}
A(\alpha)\cap\rad A(\beta)\subseteq \rad A(\alpha).
\end{equation}
Similarly, the inclusion $A(\alpha)\hookrightarrow A$
is split since $A$ is an injective $A(\alpha)$-module
and
\begin{equation}\label{eq:rad-alpha->infinity}
A(\alpha)\cap\rad A\subseteq \rad A(\alpha).
\end{equation}
Since $A(\alpha)$ is Artinian, $\rad A(\alpha)$ is
actually nilpotent and so nil, therefore $\rad A$
is also a nil ideal (in fact the largest one); see
also Lam~\cite{TYL:NonCommRings}*{proposition~4.19}.

If all of the $A(\alpha)$ are semisimple, then for
$z\in\rad A$ there must be some~$\gamma$ such that
$z\in A(\gamma)$, hence
\[
z\in A(\gamma)\cap\rad A \subseteq \rad A(\gamma)=\{0\}.
\]
Therefore $\rad A=\{0\}$ and $A$ is semiprimitive
(or Jacobson semisimple in the terminology of
Lam~\cite{TYL:NonCommRings}). In fact since each
$A(\alpha)$ is von Neumann regular, it easily
follows that~$A$ is too; see
Lam~\cite{TYL:NonCommRings}*{corollary~4.24}.
Since~$A$ is not Noetherian, it cannot be
semisimple by Lam~\cite{TYL:NonCommRings}*{corollary~4.25}.

For primitive ideals we have the following.
If $\ann_A(S)$ is the annihilator of a simple
$A$-module~$S$ which restricts to a simple
$A(\alpha)$-module for some~$\alpha$, then
for any $\beta\succcurlyeq\alpha$,
\[
A(\beta)\cap\ann_A(S) = \ann_{A(\beta)}(S).
\]

\begin{prop}\label{prop:Local}
Suppose that each $A(\alpha)$ is local with
maximal ideal~$A(\alpha)^+$, so that each
homomorphism $A(\alpha)\to A(\beta)$ is
local. Then~$A$ is also local.
\end{prop}
\begin{proof}
The kernel of the augmentation $\epsilon\:A\to\k$
is the maximal ideal~$A^+\lhd A$ for which
\[
A(\alpha)\cap A^+
\subseteq A(\alpha)^+
=\rad A(\alpha).
\]
Now given $z\in A^+$ and $a\in A$ we can assume
that $z,a\in A(\gamma)$ for some~$\gamma$ and
so $1+az$ has a left inverse in $A(\gamma)\subseteq A$;
a similar argument shows that $1+za$ has a right
unit, therefore~$A^+=\rad A$.
\end{proof}

\subsection*{Annihilators}
We will discuss left annihilators and modules,
but similar considerations apply to right
annihilators and modules.

Let $a\in A$. Then the left $A$-module
homomorphism
\[
A\to Aa;\quad x\mapsto xa
\]
fits into a short exact sequence
\[
0\to \annl_A(a)\to A\to Aa\to 0
\]
where $Aa\subseteq A$ is a \fg. submodule
and so \fp., therefore $\annl_A(a)$ is
also a \fp. module. More generally we have
\begin{prop}\label{prop:ann-fp}
Let $M$ be a coherent $A$-module and~$m\in M$.
Then the left ideal\/ $\ann_A(m)\subseteq A$
is a \fp. submodule, hence it is a coherent
$A$-module.

More generally, if\/ $W\subseteq M$ is a \fd.\/
$\k$-subspace then its annihilator $\ann_A(W)$
is \fg. as a left ideal, hence~$\ann_R(W)$
and $M/\ann_R(W)$ are coherent modules.
\end{prop}
\begin{proof}
For the first part,
see~\cite{TYL:LectModules&Rings}*{theorem~2.4.58}.

If $n_1,\ldots,n_k$ span the $\k$-vector space~$W$,
then
\[
\ann_R(W)=\ann_R(n_1)\cap\cdots\cap\ann_R(n_k).
\]
By a well-known argument, the intersection
of two \fg. submodules of a coherent module
is \fg. and so coherent. Hence,~$\ann_R(W)$
and $M/\ann_R(W)$ are coherent modules.
\end{proof}


\subsection*{Finite and simple modules}
Although~$A$ will have simple modules, in general
they will not all be coherent.
%
%
Frobenius algebras are Kasch algebras (i.e.,
every left/right simple module is isomorphic
to a left/right ideal). The situation for
a locally Frobenius algebra~$A$ is less
straightforward. For example, the trivial
module~$\k$ can never be isomorphic to a
submodule because of the next result
characterising the simple modules which
occur as minimal left/right ideals of~$A$.

\begin{prop}\label{prop:SimpleMod-coherent}
Let $S$ be a simple $A$-module. Then $S$
is isomorphic to submodule of $A$ if and
only if it is coherent.
\end{prop}
\begin{proof}
Assume that $S\subseteq A$ and $0\neq s\in S$.
Since $As=S$, Proposition~\ref{prop:ann-fp}
implies that there is a short exact
sequence
\[
0\to \annl_A(s)\to A\to S\to 0
\]
with $\annl_A(s)$ \fp., so $S$ is
coherent.

For the converse, suppose that $S$
is a coherent simple $A$-module. By
Proposition~\ref{prop:fp-injective},
there is an embedding $j\:S\hookrightarrow F$
into a \fg. free module. At least one of the
compositions of~$j$ with the projection onto
a copy of~$A$ must be non-zero and so a
monomorphism by simplicity, hence~$S$ is
isomorphic to a submodule of~$A$.
\end{proof}

We already know that \fd. simple modules are
not coherent by Corollary~\ref{cor:fd-simple-noncoherent},
so combining this with
Proposition~\ref{prop:SimpleMod-coherent}
we obtain an important consequence.
\begin{cor}\label{cor:fdSimple-notinA}
A \fd. simple module $S$ is not isomorphic to
a submodule of~$A$, or equivalently~$\Hom_A(S,A)=0$.
More generally, for a coherent $A$-module~$N$,
$\Hom_A(S,N)=0$.

In particular this applies to the trivial
module\/~$\k$.
\end{cor}
\begin{proof}
For the second statement, recall that there
is an embedding~$N\hookrightarrow F$ into
a finitely generated free module. The short
exact sequence
\[
0\to N\to F\to F/N\to0
\]
induces a long exact sequence beginning
with
\[
\xymatrix@C=0.8cm@R=0.2cm{
0\ar[r]&\Hom_A(S,N)\ar[r]&\Hom_A(S,F)\ar[r]\ar@{=}[d]
&\Hom_A(S,F/N)\ar[r]&\cdots \\
&&0&&
}
\]
so $\Hom_A(S,M)=0$.
\end{proof}

We can extend this result to arbitrary \fd.
modules.
\begin{cor}\label{cor:fd->Coherent}
Suppose that~$M$ is a \fd. $A$-module. Then
for any coherent module~$N$,
\[
\Hom_A(M,N) = 0.
\]
\end{cor}
\begin{proof}
This can be proved by induction on the length
of a composition series of~$M$. The inductive
step uses the fact that for a \fd. simple
module~$S$, we have~$\Hom_A(S,N)=0$ by
Propositions~\ref{prop:SimpleMod-coherent}
and~\ref{prop:fp-injective}.
\end{proof}


We also have the following.
\begin{prop}\label{prop:CoherentMinIdeal}
Suppose that $L\subseteq A$ is a minimal
left ideal and~$0\neq s\in L$. Then at
least one of\/ $L\subseteq A^+$ or\/
$\annl_A(s)\subseteq A^+$ must be true.
Furthermore, if $\annl_A(s)\subseteq A^+$
then\/~$\dim_\k L=\infty$.

Similar results are true for minimal right
ideals.
\end{prop}
\begin{proof}
By Proposition~\ref{prop:SimpleMod-coherent},
$\annl_A(s)$ is \fg., and since $\annl_A(s)s=0$
we have $s\in A^+$ or $\ann_A(s)\subseteq A^+$
because~$A^+$ is completely prime. By
Proposition~\ref{prop:cohmod-infdim}, if~$\ann_A(s)\subseteq A^+$
then~$L\iso A/\ann_A(s)$ is infinite dimensional.
\end{proof}

\begin{prop}\label{prop:Simplerestricted}
Suppose that~$S$ is a finite dimensional
simple $A$-module. Then there is a
$\lambda\in\Lambda$ for which the restriction
of $S$ to $A(\lambda)$ is simple. Hence for
$\lambda'\in\Lambda$ with $\lambda\preccurlyeq\lambda'$,
the restriction of $S$ to $A(\lambda')$
is also simple.
\end{prop}
\begin{proof}
I would like to thank Ken Brown for the following
proof.

The quotient algebra $A/\ann_A(S)$ is a finite
dimensional simple $\k$-algebra with dimension
$\dim_\k A/\ann_A(S)=m$ say. Lift a $\k$-basis
of $A/\ann_A(S)$ to a linearly independent set
$E=\{e_1,\ldots,e_m\}\subseteq A$ which is
contained in some $A(\lambda)$. We also have
the ideal
\[
J=A(\lambda)\cap\ann_A(S)\lhd A(\lambda)
\]
and the composition
\[
A(\lambda)\hookrightarrow A\to A/\ann_A(S)
\]
factors through an injective homomorphism
$A(\lambda)/J\to A/\ann_A(S)$. Since the
images of the elements of~$E$ in
$A(\lambda)/J$ are still linearly independent,
we have $\dim_\k(A(\lambda)/J)\geq m$; but
as $\dim_\k A(\lambda)/J\leq m$ we obtain
$\dim_\k A(\lambda)/J=\dim_\k A/\ann_A(S)$.
Therefore we have $A(\lambda)/J\iso A/\ann_A(S)$,
making a $A(\lambda)/J$ simple Artinian ring.
But we know that~$S$ is the (unique) simple
$A/\ann_A(S)$-module so it is also simple
as an $A(\lambda)/J$-module and as an
$A(\lambda)$-module.

The other statement is clear.
\end{proof}

%

\begin{cor}\label{cor:fdsimpnoncoh}
Suppose that $S$ is a \fd. non-coherent
simple $A$-module. Then there is an
$\alpha\in\Lambda$ such that
\begin{itemize}
\item
$S$ is a simple $A(\alpha)$-module;
\item
there is an embedding of $A(\alpha)$-modules
$S\hookrightarrow A(\alpha)$;
\item
there is a $\beta\succcurlyeq\alpha$ such
that the composition
$S\hookrightarrow A(\alpha)\hookrightarrow A(\beta)$
is not a homomorphism of $A(\beta)$-modules.
Therefore there is an $a\in A(\beta)$ such
that $aS\nsubseteq S$ and in particular
$aS\setminus S\neq\varnothing$.
\end{itemize}
\end{cor}
\begin{proof}
The first statement follows from
Proposition~\ref{prop:Simplerestricted}
and the second is a consequence of a Frobenius
algebras being a Kasch algebra. If for every
$\beta\succcurlyeq\alpha$ this composition
were a $A(\beta)$-module homomorphism then
we would obtain an $A$-module embedding
$S\hookrightarrow A$, contradicting
Proposition~\ref{prop:SimpleMod-coherent}.
\end{proof}

Of course there may be infinite dimensional
minimal left/right ideals which satisfy one
of the conditions~$L^2=0$ or~$L^2=L$. If~$L=Az$
and $L^2=0$ then~$z\in A^+$ since this ideal
is completely prime.

The next results were suggested by analogous
results of
Richardson~\cite{JSR:GpRngsnonzerosocle}.
\begin{lem}\label{lem:MinIdeal}
Let $L$ be a minimal left/right ideal in~$A$.
Then there is a $\lambda\in\Lambda$ for which
$L(\lambda)=A(\lambda)\cap L$ is a non-trivial
minimal left/right ideal in~$A(\lambda)$ and
\[
L = AL(\lambda)\iso A\otimes_{A(\lambda)}L(\lambda).
\]
\end{lem}
\begin{proof}
We outline the proof for left ideals.

Write $L=Az$ and choose $\lambda\in\Lambda$ so
that $z\in A(\lambda)$. Then $L(\lambda)=A(\lambda)\cap L$
is a non-trivial left ideal in~$A(\lambda)$ and
$AL(\lambda)=L$. By the well-known criterion for
flatness of \cite{TYL:LectModules&Rings}*{(4.12)},
multiplication gives an isomorphism
\[
A\otimes_{A(\lambda)}L(\lambda)\xrightarrow{\iso}AL(\lambda)=L
\]
and Corollary~\ref{cor:SimpleCoherent} implies
that~$L(\lambda)$ is simple.
\end{proof}
\begin{cor}\label{cor:MinIdeal}
Let $L$ be a minimal left ideal in~$A$. Then either
$L^2=0$ or there is an idempotent $e\in A$ such
that $L=Ae$.

Similarly for a minimal right ideal in~$A$, either
$L^2=0$ or $L$ is generated by an idempotent.
\end{cor}
\begin{proof}
If $L^2\neq0$ then suppose $L(\lambda)=A(\lambda)\cap L\neq0$.
In the Artinan ring $A(\lambda)$, the minimal
ideal $L(\lambda)=L(\lambda)^2$ is generated
by an idempotent~$e$ and $L=Ae$.
\end{proof}

Our next result and its proof are direct
generalisations of
Richardson~\cite{JSR:GpRngsnonzerosocle}*{proposition~2.8}.
\begin{prop}\label{prop:Minleftideal->right}
Let $0\neq w\in A$. Then~$Aw$ is a minimal left
ideal of~$A$ if and only if~$wA$ is a minimal
right ideal. Hence the left and right socles
of~$A$ agree.
\end{prop}
\begin{proof}
If $Aw$ is a minimal left ideal then by
Lemma~\ref{lem:MinIdeal} there is a
$\lambda\in\Lambda$ such that
$w\in A(\lambda)$ and $A(\lambda)\cap Aw$
is a minimal left ideal in $A(\lambda)$.
Since~$wA$ is the union of the $wA(\lambda)$
over all such~$\lambda$, it suffices to
show that $wA(\lambda)$ is a minimal right
ideal. There is an isomorphism of left
$A(\lambda)$-modules
\[
A(\lambda)w \xrightarrow{\iso} A(\lambda)/\annl_{A(\lambda)}(w)
\]
so $\annl_{A(\lambda)}(w)=\ann_{A(\lambda)}(wA(\lambda))$
which is a maximal left ideal. As~$A(\lambda)$
is Frobenius its left and right submodule
lattices correspond under an inclusion
reversing bijection, therefore~$wA(\lambda)$
is the annihilator ideal of~$A(\lambda)w$
which is a  minimal right ideal.
\end{proof}

Another result of
Richardson~\cite{JSR:GpRngsnonzerosocle}*{lemma~2.7}
(see also Hartley \& Richardson~\cite{BH&JSR:SocleGpRngs})
can be extended to a locally Frobenius algebra.
We know that a minimal left ideal of~$A$ is a
coherent $A$-module and so infinite dimensional;
nevertheless this result gives some sort of
finiteness albeit over a division algebra
which is infinite dimensional over~$\k$.
\begin{prop}\label{prop:Richardson-lemma2.7}
Let $L=Aw$ be a minimal left ideal of $A$. \\
\emph{(a)} The ideal~$L$ is finite dimensional
over the division algebra\/~$\End_A(L)$. \\
\emph{(b)} Every finite subset of\/~$\End_A(L)$
lies in a finite dimensional separable\/
$\k$-subalgebra. \\
\emph{(c)} If\/ $\Char\k>0$, then\/~$\End_A(L)$
is a field.
\end{prop}
\begin{proof}
This involves a routine reworking of that for
locally finite group algebras and is left to
the reader, the main point to notice is that
group algebras of finite subgroups need to be
replaced by Frobenius algebras~$A(\lambda)$.
\end{proof}

For locally Frobenius Hopf algebras our results
on \fd. modules should be compared with the
following.
\begin{prop}[See~\cite{ML:TourRepThy}*{proposition~10.6}]
\label{prop:HopfAlg-fdideal->infdim}
Suppose that~$H$ is a Hopf algebra over\/ $\k$.
If~$H$ is infinite dimensional then it has no
non-trivial \fd. left/right ideals.
\end{prop}
Combining this with Proposition~\ref{prop:SimpleMod-coherent}
we obtain a result which also follows from
Proposition~\ref{prop:cohmod-infdim}.
\begin{cor}\label{cor:HopfAlg-fdideal->infdim}
If $H$ is a locally Frobenius Hopf algebra
then no \fd. simple module is coherent.
\end{cor}

\subsection*{Coherent projective covers}

In general \fg. modules over a coherent ring
may not have projective covers in the usual
sense. But we can define something similar
for coherent modules over a locally Frobenius
algebra~$A$.

Consider the coherent $A$-module
$A\otimes_{A(\alpha)}M_\alpha$ where $M_\alpha$
is a \fg. (hence \fd.) $A(\alpha)$-module.
If~$\alpha\preccurlyeq\beta$ then
\[
M_\beta = A(\beta)\otimes_{A(\alpha)}M_\alpha
\]
has dimension
\[
\dim_\k M_\beta
=
\frac{\dim_\k A(\beta)\dim_\k M_\alpha}{\dim_\k A(\alpha)}.
\]
so
\[
\frac{\dim_\k M_\beta}{\dim_\k A(\beta)}
=
\frac{\dim_\k M_\alpha}{\dim_\k A(\alpha)}.
\]
By Corollary~\ref{cor:set-up-SES}, if there
is an isomorphism of $A$-modules
$A\otimes_{A(\alpha)}M_\alpha
\iso
A\otimes_{A(\beta)}N_\beta$
then for some large enough $\gamma$ there
is an isomorphism of $A(\gamma)$-modules
\[
M_\gamma=A(\gamma)\otimes_{A(\alpha)}M_\alpha
\iso A(\gamma)\otimes_{A(\beta)}N_\beta=N_\gamma,
\]
therefore
\[
\frac{\dim_\k M_\gamma}{\dim_\k A(\gamma)}
=
\frac{\dim_\k N_\gamma}{\dim_\k A(\gamma)}.
\]
This shows that to any coherent $A$-module
we can assign an isomorphism invariant
rational number by taking its \emph{coherent
dimension} to be
\[
\cohdim(A\otimes_{A(\alpha)}M_\alpha)
=
\frac{\dim_\k M_\alpha}{\dim_\k A(\alpha)}.
\]

Now let $M$ be a coherent $A$-module. Then
there exist epimorphisms of $A$-modules
$P\to M$ where~$P$ is a coherent projective
module. By Corollary~\ref{cor:A-reflectsepimono},
such an epimorphism can be taken to be induced
up from an epimorphism of~$A(\alpha)$-modules
$P_\alpha\to M_\alpha$ for some $\alpha\in\Lambda$,
where~$P_\alpha$ is projective; we will refer to
such epimorphisms as \emph{coherent epimorphisms}.
Notice also that $\dim_\k P_\alpha\geq\dim_\k M_\alpha$
so
\[
\cohdim(A\otimes_{A(\alpha)}P_\alpha)
\geq\cohdim(A\otimes_{A(\alpha)}M_\alpha).
\]

In fact we have the following useful observation.
Let $\phi\:P\to N$ be a homomorphism of coherent
$A$-modules with~$P$ coherent projective,
and let $\theta\:M\to N$ be a coherent epimorphism.
Then we can factor~$\phi$ through~$M$.
\[
\xymatrix{
& M\ar[d]^\theta \\
P\ar[r]^\phi\ar@/^15pt/[ur]^{\tilde\phi} & N
}
\]
By choosing a large enough $\alpha\in\Lambda$
we can find underlying homomorphisms of \fg.
$A(\alpha)$-modules
\[
\xymatrix{
& M_\alpha\ar[d]^{\theta_\alpha} \\
P_\alpha\ar[r]^{\phi_\alpha}\ar@/^15pt/[ur]^{\tilde\phi_\alpha}
& N_\alpha
}
\]
which on applying $A\otimes_{A(\alpha)}(-)$ induce
the diagram of $A$-modules. We need to check that
this second diagram commutes. Since
$\theta_\alpha\circ\tilde\phi_\alpha$ and
$\phi_\alpha$ induce the same homomorphism $P\to M$,
so by Proposition~\ref{prop:FF}(a)
$\theta_\alpha\circ\tilde\phi_\alpha-\phi_\alpha$
must be trivial and this diagram also commutes.

Now suppose that we have two coherent epimorphisms
$P\xrightarrow{\phi} M\xleftarrow{\theta} Q$ where~$P$
and~$Q$ are coherent projectives. The preceding
discussion shows that we can assume that for some
$\lambda\in\Lambda$ there are there are projective
$A(\lambda)$-modules $P_\lambda$ and $Q_\lambda$
for which $P\iso A\otimes_{A(\lambda)}P_\lambda$
and $Q\iso A\otimes_{A(\lambda)}Q_\lambda$ together
with a \fg. $A(\lambda)$-module $M_\lambda$ for
which $M\iso A\otimes_{A(\lambda)}M_\lambda$;
furthermore there is a commutative diagram of
$A(\lambda)$-modules
\[
\xymatrix{
P_\lambda\ar@{->>}@/_10pt/[dr]_{\phi_\lambda}\ar[r]^\rho
& Q_\lambda\ar@{->>}[d]_{\theta_\lambda}\ar[r]^\sigma
& P_\lambda\ar@{->>}@/^10pt/[dl]^{\phi_\lambda}  \\
& M_\lambda &
}
\]
where $\phi_\lambda$ and $\theta_\lambda$ induce
$\phi$ and~$\theta$. On tensoring with~$A$ and
using faithful flatness we see that
$\sigma\circ\rho=\Id-{P_\lambda}$, so~$P_\lambda$
is a summand of~$Q_\lambda$ and
$\dim_\k P_\lambda\geq\dim_\k Q_\lambda$.

We can now define the \emph{coherent rank of~$M$}
to be
\[
\cohrk M =
\min\{\cohdim P :
\text{there is a coherent epi. $P\to M$}\}
\geq \cohdim(M).
\]
If we realise this minimum with a coherent epimorphism
$P\to M$, then for any other such coherent epimorphism
$Q\to M$ the above discussion shows that $\cohdim P=\cohdim Q$
and so~$P\iso Q$. Thus we can regard $P\to M$ as a
\emph{coherent projective cover}.


\subsection*{Pseudo-coherent modules}

In Bourbaki~\cite{Bourbaki:HomAlg}*{ex.~\S3.10},
as well as coherent modules, pseudo-coherent
modules are considered, where a module is
\emph{pseudo-coherent} if every finitely
generated submodule is finitely presented
(so over a coherent ring it is coherent).
The reader is warned that pseudo-coherent
is sometimes used in a different sense; for
example, over a coherent ring the definition
of Weibel~\cite{CAW:Ktheory}*{example~II.7.1.4}
corresponds to our coherent.

Examples of pseudo-coherent modules over a coherent
ring~$A$ include coproducts of coherent modules
such as the \fr. modules of Definition~\ref{defn:fg-fr-fp}.
However pseudo-coherent modules do not form a full
abelian subcategory of~$\Mod_A$ since for example
cokernels of homomorphisms between pseudo-coherent
modules need not be pseudo-coherent. Nevertheless,
pseudo-coherent modules do occur quite commonly
when working with coherent modules over locally
Frobenius algebras.

\begin{lem}\label{lem:pc-colimit}
Let $A$ be a ring and $M$ a pseudo-coherent
$A$-module. Then~$M$ is the union of its
coherent submodules and therefore their
colimit.
\end{lem}
\begin{proof}
Every element~$m\in M$ generates a cyclic
submodule which is coherent.
\end{proof}

\begin{lem}\label{lem:pc-extended}
Let $A$ be a coherent ring and $B$ an Artinian
subring where~$A$ is flat as a right $B$-module.
Then every extended $A$-module $A\otimes_BN$
is pseudo-coherent.
\end{lem}
\begin{proof}
Let $U\subseteq A\otimes_BN$ be a finitely
generated submodule. Taking a finite generating
set and expressing each element as a sum of basic
tensors we find that $U\subseteq A\otimes_BU'$
where $U'\subseteq N$ is a finitely generated
submodule. As~$B$ is Artinian and so Noetherian,~$U'$
is finitely presented, so by flatness of~$A$,
$A\otimes_BU'$ is a finitely presented $A$-module.
Since $A$ is coherent,~$U$ is also finitely
presented.
\end{proof}

An important special case of this occurs when~$A$
is a coherent algebra over a field~$\k$ and~$B$
is a finite dimensional subalgebra. If~$N$ is
a $B$-module then it is locally finite, i.e.,
every element is contained in a finite
dimensional submodule.

We will discuss pseudo-coherence for modules
over a locally Frobenius Hopf algebra in
Section~\ref{sec:LFHA}.

The following stronger notions of pseudo-coherence
are perhaps more likely to be important for example
for computational purposes.
\begin{defn}\label{defn:spc}{\ }
\begin{itemize}
\item
A module over a coherent ring is \emph{strongly
pseudo-coherent\/} if it is a coproduct of
coherent modules.
\item
A module $M$ over a locally Frobenius algebra~$A$
indexed on~$\Lambda$ is \emph{$\lambda$-strongly
pseudo-coherent} for $\lambda\in\Lambda$ if
\[
M\iso A\otimes_{A(\lambda)}M'
\]
where the $A(\lambda)$-module~$M'$ is a coproduct
of finitely generated $A(\lambda)$-modules.
\end{itemize}
\end{defn}

Over a coherent ring every free module is strongly
pseudo-coherent as is every finitely related
module since it is the sum of a coherent module
and a free module. Clearly a $\lambda$-strongly
pseudo-coherent module over a locally Frobenius
algebra~$A$ is strongly pseudo-coherent.

\subsection*{Cohomology for finite dimensional
$A$-modules}
To end this section we discuss $\Ext$ groups
for finite dimensional modules. We will work
with left modules, but a similar discussion
applies to right modules.

If $W$ is an $A(\lambda)$-module for some
$\lambda\in\Lambda$ and~$N$ is an $A$-module
(hence an $A(\lambda)$-module via restriction),
then $A\otimes_{A(\lambda)}W$ is an $A$-module
and
\[
\Hom_{A(\lambda)}(W,N) \iso
\Hom_{A(\lambda)}(A\otimes_{A(\lambda)}W,N).
\]
More generally, since $A$ is $A(\lambda)$-flat,
\[
\Ext^s_{A(\lambda)}(W,N) \iso
\Ext^s_{A}(A\otimes_{A(\lambda)}W,N).
\]
If $M,N$ are $A$-modules viewed as
$A(\lambda)$-modules through restriction,
then
\[
\colim_{(\Lambda,\preccurlyeq)}A\otimes_{A(\lambda)}M
\iso M
\]
as $A$-modules and
\begin{equation}\label{eq:SS-SES-Hom}
\Hom_{A}(M,N) \iso
\Hom_{A}(\colim_{(\Lambda,\preccurlyeq)}A\otimes_{A(\lambda)}M,N)
\iso \lim_{(\Lambda,\preccurlyeq)}\Hom_{A(\lambda)}(M,N).
\end{equation}

There is a spectral sequence of
Jensen~\cite{LNM254}*{th\'eor\`eme~4.2} of the
form
\[
\mathrm{E}_2^{s,t} =
{\lim_{(\Lambda,\preccurlyeq)}}^{\!s}\Ext^t_{A}(A\otimes_{A(\lambda)}M,N)
\Lra \Ext^{s+t}_{A}(M,N),
\]
where $\ds{\lim_{(\Lambda,\preccurlyeq)}}^{\!s}$
is the $s$-th right derived functor of
$\ds\lim_{(\Lambda,\preccurlyeq)}$. The
$\mathrm{E}_2$-term can be rewritten to
give
\begin{equation}\label{eq:SS}
\mathrm{E}_2^{s,t} =
{\lim_{(\Lambda,\preccurlyeq)}}^{\!s}\Ext^t_{A(\lambda)}(M,N)
\Lra \Ext^{s+t}_{A}(M,N).
\end{equation}
When $\Lambda$ is countable, by a result
of Jensen~\cite{LNM254}, for~$s>1$,
$\ds{\lim_{\lambda\in\Lambda}}^{\!s}$ is
trivial, so for each $n\geq1$ there is an
exact sequence
\begin{equation}\label{eq:SS-SES-Ext}
0 \to
{\lim_{(\Lambda,\preccurlyeq)}}^{\!1}\Ext^{n-1}_{A(\lambda)}(M,N)
\to \Ext^{n}_{A}(M,N) \to
\lim_{(\Lambda,\preccurlyeq)}\Ext^{n}_{A(\lambda)}(M,N).
\to 0
\end{equation}

Suppose that $J$ is an injective $A$-module
and
\[
\xymatrix{
0 \ar[r] & U\ar[r]\ar[d] & V  \\
 & J &
}
\]
is a diagram of $A(\lambda)$-modules with
an exact row. By flatness of~$A$ as an
$A(\lambda)$-module,
\[
\xymatrix{
0 \ar[r] & A\otimes_{A(\lambda)}U\ar[r]
& A\otimes_{A(\lambda)}V
}
\]
is an exact sequence of $A$-modules and
on applying $\Hom_{A(\lambda)}(-,J)$ we
obtain a commutative diagram
\[
\xymatrix{
& \Hom_{A(\lambda)}(U,J)\ar@{<->}[d]_\iso
& \ar[l]\Hom_{A(\lambda)}(V,J)\ar@{<->}[d]^\iso  \\
0 & \ar[l]\Hom_{A}(A\otimes_{A(\lambda)}U,J)
& \ar[l]\Hom_{A}(A\otimes_{A(\lambda)}V,J)
}
\]
where the lower row is exact by injectivity
of the $A$-module~$J$. It follows that~$J$
is an injective $A(\lambda)$-module.

Now taking $N=J$ and using \eqref{eq:SS-SES-Ext}
with $n=1$ we obtain
\[
{\lim_{(\Lambda,\preccurlyeq)}}^{\!1}\Hom_{A(\lambda)}(M,J)
= 0,
\]
and by using~\eqref{eq:SS-SES-Hom},
\[
\Hom_{A}(M,J) \iso
\lim_{(\Lambda,\preccurlyeq)}\Hom_{A(\lambda)}(M,J).
\]

We know that $A$ is relatively injective with
respect to coherent $A$-modules. For a finite
dimensional non-coherent $A$-module~$M$, the
spectral sequence~\eqref{eq:SS} gives
\[
\mathrm{E}_2^{s,t} =
{\lim_{(\Lambda,\preccurlyeq)}}^{\!s}\Ext^t_{A(\lambda)}(M,A)
\Lra \Ext^{s+t}_{A}(M,A)
\]
and as $A$ is $A(\lambda)$-injective for every
$\lambda$, while if $t=0$,
\[
\mathrm{E}_2^{s,t} =
\begin{dcases*}
{\lim_{(\Lambda,\preccurlyeq)}}^{\!s}\Hom_{A(\lambda)}(M,A) & if $t=0$, \\
\quad\quad\quad\quad\quad 0  & if $t>0$.
\end{dcases*}
\]
So we have
\[
\Ext^n_{A}(M,A) \iso
{\lim_{(\Lambda,\preccurlyeq)}}^{\!n}\Hom_{A(\lambda)}(M,A).
\]
In particular, if $\Lambda$ is countable,
$\Ext^n_{A}(M,A)=0$ when~$n>1$.
\begin{defn}\label{defn:S-good}
If $S$ is a finite dimensional simple $A$-module
then $A$ is \emph{$S$-Margolisian} if $\Ext^n_{A}(S,A)=0$
for all~$n>0$.

If $A$ is $S$-Margolisian for all finite dimensional
simple $A$-modules then~$A$ is~\emph{Margolisian}.
\end{defn}

So $A$ is Margolisian if for every finite dimensional
non-coherent simple $A$-module $S$ and every~$s>0$,
\[
{\lim_{(\Lambda,\preccurlyeq)}}^{\!s}\Hom_{A(\lambda)}(M,A)
= 0.
\]
We can summarise this in a result which is
obtained by combining these ideas with what
we already know for coherent finite dimensional
$A$-modules.
\begin{prop}\label{prop:Good}
$A$ is Margolisian if and only if $A$ is
relatively injective with respect to the
abelian category\/ $\Mod_A^{\mathrm{f.d.}}$
of all finite dimensional $A$-modules.
\end{prop}

Margolis~\cite{HRM:Book} shows that every
$P$-algebra is $\k$-Margolisian in our sense.
In this case~$\k$ is the only simple module.
His proof makes essential use of the fact
that its modules are graded and we have
neither been able to find an argument that
works in our situation nor a counter example,
so it seems possible that every locally
Frobenius algebra is Margolisian, or at
least $\k$-Margolisian.

\begin{rem}\label{rem:fd-Coherent-Ext}
As in Remark~\ref{rem:Coherent-Ext}, for a
\fd. $A$-module~$M$, the left exact functor
$\Hom_A(M,-)$ on the category of all $A$-modules
has the right derived functors $\Ext_A^*(M,-)$.
If $A$ is Margolisian then for each finitely
generated free module~$F$, $\Ext_A^*(M,F) = 0$
and~$F$ is $\Hom_A(M,-)$-acyclic, so these right
derived functors can be computed using resolutions
by such modules which are the injectives in
$\Mod_A^{\mathrm{coh}}$. So they are also right
derived functors on $\Mod_A^{\mathrm{coh}}$.
But then $\Ext_A^*(M,N)=0$ for every coherent
module~$N$.
\end{rem}

We now discuss the finite dual of a locally
Frobenius $\k$-algebra~$A$ and some related
ideas. As we could not find a convenient
reference we give details which are probably
well known.

Recall that the \emph{finite dual} of~$A$ is
\[
A^\o =
\{ f\in\Hom_\k(A,\k) :
\text{$\exists I\lhd A$ s.t. $I$ is cofinite
and $I\subseteq\ker f$} \}
\subseteq\Hom_\k(A,\k).
\]
Here $I$ is \emph{cofinite} if $\dim_\k A/I<\infty$.
It is clear that~$A^\o\subseteq\Hom_\k(A,\k)$ is
a $\k$-subspace.

There are two $A$-module structures on~$\Hom_\k(A,\k)$
namely a left one induced by right multiplication
on the domain and a right one induced by left
multiplication on the domain.
\begin{lem}\label{lem:finitedualmodule}
The left and right $A$-module structures on~$A^\o$
each restrict to make~$A^\o$ a submodule of~$\Hom_\k(A,\k)$.

Furthermore, if $M$ is a \fd. $A$-module the image
of every $A$-module homomorphism $M\to\Hom_\k(A,\k)$
is contained in~$A^\o$.
\end{lem}
\begin{proof}
We give the argument for the left module structure.

Let $f\in A^0$ be trivial on a cofinite ideal~$I\lhd A$.
If $a\in A$ and $x\in I$ then
\[
(af)(x) = f(xa) = 0
\]
since $xa\in I$.

Since $M$ is \fd., the associated algebra homomorphism
$A\to\End_\k(M)$ has a cofinite kernel $J\lhd A$ say.
Now for an $A$-module homomorphism $M\to\Hom_\k(A,\k)$,
the image of each $m\in M$ must vanish on~$J$.
\end{proof}

If $W$ is a \fd. vector space then we can similarly
define
\[
\Hom_\k^{\mathrm{fin}}(A,W) =
\{ f\in\Hom_\k(A,W) :
\text{$\exists I\lhd A$ s.t. $I$ is cofinite and
$I\subseteq\ker f$} \}
\subseteq\Hom_\k(A,W).
\]
Then the obvious left and right $A$-module structures
on $\Hom_\k(A,W)$ also restrict to make
$\Hom_\k^{\mathrm{fin}}(A,W)$ a left or right submodule.
If we choose a basis for~$W$ we obtain isomorphisms
of left and right $A$-modules
\[
\Hom_\k^{\mathrm{fin}}(A,W) \iso (A^\o)^{\dim_\k W}.
\]
An analogue of this is also true when~$W$ is infinite
dimensional.

It is standard that injectives in the module category
$\Mod_A$ are summands of modules of the form $\Hom_\k(A,W)$.
In particular every $A$-module~$M$ admits an embedding
$M\hookrightarrow\Hom_\k(A,M)$ where the copy of~$M$
in the codomain is viewed just as a vector space.
\begin{lem}\label{lem:findual-relinj}
Suppose that~$W$ is a \fd. vector space. Then\/
$\Hom_\k^{\mathrm{fin}}(A,W)$ is a relative injective
with respect to the category of \fd. $A$-modules.
Furthermore every \fd. $A$-module admits a resolution
by such relative injectives.
\end{lem}
\begin{proof}
It suffices to show this for $A^\o$ itself.

Suppose we have the solid diagram of $A$-modules
\[
\xymatrix{
0\ar[r] & U\ar[r]\ar[d] & V\ar@{-->}[ld]\ar@/^15pt/@{..>}[ldd] \\
& A^\o\ar@{^{(}->}[d] &  \\
& \Hom_\k(A,\k) &
}
\]
where $U$ and $V$ are \fd. and the row is exact.
Because $\Hom_\k(A,\k)$ is a genuine injective
$A$-module there is a dotted arrow making the
diagram commute. By the second part of
Lemma~\ref{lem:finitedualmodule} its image is
contained in~$A^\o$ so we obtain a dashed arrow
making the diagram commute.
\end{proof}

This means that to calculate $\Ext_A^*(M,N)$ where
$M$ and $N$ are \fd., we can use a resolution of
$N$ by $A$-modules of the form
$\Hom_\k^{\mathrm{fin}}(A,W)$. Of course this is
not a resolution in $\Mod_A^{\mathrm{\fd.}}$, but
it is in the category of \emph{locally finite}
$A$-modules $\Mod_A^{\mathrm{\lf.}}$ which is a
full abelian subcategory of~$\Mod_A$.

\section{Stable module categories for locally
Frobenius algebras}\label{sec:StabModCat}


For an arbitrary ring there are two distinct ways
to define a stable module category starting with
the category of modules, namely by treating as
trivial those homomorphisms which factor through
either projectives or injectives; for a Frobenius
algebra these coincide. A locally Frobenius
algebra~$A$ is not self-injective, but it is
injective relative to certain types of $A$-modules,
in particular coherent modules and \fr. modules.
%
So starting with the abelian category of coherent
$A$-modules we obtain a stable module category
using either approach. However, there are grounds
for thinking that \fr. or pseudo-coherent modules
should also be included although neither form an
abelian category in an obvious way.

We adopt some standard notions for stable module
categories for the stable module category of
coherent $A$-modules which we will denote by
$\Stmod^{\mathrm{coh}}_A$. We will also use the
notation
\[
\{M,N\} = \Stmod^{\mathrm{coh}}_A(M,N)
= \Mod^{\mathrm{coh}}_A(M,N)/\approx
\]
for the morphisms, where for two homomorphisms
$f,g\:M\to N$, $f\approx g$ if and only if there
is a factorisation
\[
\xymatrix{
M\ar[rr]^{f-g}\ar[dr] && N \\
& F\ar[ur] &
}
\]
where $F$ is a \fg. free module. For a coherent
module~$M$, the objects $\Omega M$ and $\mho M$
are defined by taking exact sequences
\[
P\xrightarrow{p}M \to 0,\quad 0\to M\xrightarrow{j}J
\]
where $P$ and $J$ are \fg. free and hence injective
modules. Using Schanuel's Lemma,~$\ker p$ and
$\coker j$ are well-defined up to isomorphism
in $\Stmod^{\mathrm{coh}}_A$ so we may use
these as representatives of~$\Omega M$ and~$\mho M$.

\begin{lem}\label{lem:Omega-mho}
In\/ $\Stmod^{\mathrm{coh}}_A$, every object~$M$
is isomorphic to objects of the form $\Omega M'$
and $\mho M''$ for some coherent modules~$M'$
and~$M''$. Then~$\Omega$ and~$\mho$ define
mutually inverse endofunctors of\/
$\Stmod^{\mathrm{coh}}_A$.
\end{lem}
\begin{proof}
By Proposition~\ref{prop:fr-embedding}(a)
we know that $M$ is a submodule of a \fg.
free module $F$ so there is an exact sequence
\[
0\to M\to F\to M'\to 0
\]
showing that $M=\Omega M'$. On the other
hand there is also an epimorphism $F'\to M$
for some \fg. free module~$F'$ and by
Lemma~\ref{lem:Ainj-fp} this is injective,
hence there is an exact sequence
\[
0\to M''\to F'\to M\to 0
\]
and therefore $M=\mho M''$. This also shows
that $M''\iso\Omega M$ and so $M\iso\mho\Omega M$.
A similar argument shows that $M\iso\Omega\mho M$.
\end{proof}

\begin{prop}\label{prop:Omega-mho}
The functors
$\Omega,\mho\:\Stmod_A^{\mathrm{coh}}\to\Stmod_A^{\mathrm{coh}}$
satisfy the following for all coherent $A$-modules
$M$ and $N$ and these isomorphisms are natural:
\[
\{\Omega M,\Omega N\} \iso \{M,N\}
\iso \{\mho M,\mho N\}.
\]
\end{prop}

We extend $\{-,-\}$ to a graded $\Z$-bifunctor
by setting
\[
\{M,N\}^n = \{M,N\}_{-n} =
\begin{dcases*}
\{M,\mho^nN\} & if $n\geq0$, \\
\{\Omega^{-n}M,N\} & if $n<0$.
\end{dcases*}
\]
This leads to a triangulated structure on
$\Stmod_A^{\mathrm{coh}}$. Of course the
point of this definition is that when~$n>0$,
\[
\Ext_A^n(M,N) \iso \{M,N\}^n.
\]

A major defect of this stable module category
in the case when~$A$ is a locally Frobenius
Hopf algebra is its lack of an obvious monoidal
structure: the tensor product~$M\otimes_\k N$
of two coherent modules with the diagonal
action need not be coherent as an $A$-module.
This problem disappears if one of the factors
is instead a \fd. module. We will discuss
this further in Section~\ref{sec:LFHA}.

\section{Locally Frobenius Hopf algebras}
\label{sec:LFHA}

In this section we consider results on locally
Frobenius Hopf algebras where the Hopf structure
plays a r\^ole. We are particularly motivated
by the goal of extending results on locally
finite group algebras. A convenient source
for background material on Hopf algebras is
provided by Montgomery~\cite{SM:HopfAlgActions}.

\begin{assump}\label{assump:TensorProds}
We will assume that~$H$ is a Hopf algebra
over the field~$\k$ with coproduct
$\psi\:H\to H\otimes H$ and antipode
$\chi\:H\to H$, where will use the notation
for these:
\[
\psi(h) = \sum_i h'_i\otimes h_i'',
\quad
\bar{h} = \chi(h).
\]
We will also assume that $K\subseteq H$ is
a subHopf algebra with~$H$ being flat as
a left and right $K$-module.
\end{assump}

Now given two left $H$-modules $L$ and~$M$,
their tensor product~$L\otimes M$ is a left
$H$-module with the diagonal action given by
\[
h\cdot(\ell\otimes m)
= \sum_i h'_i\ell\otimes h''_im,
\]
where the coproduct on~$h$ is
\[
\psi h = \sum_i h'_i\otimes h''_i.
\]

In particular, given a left $H$-module~$M$ and
a left $K$-module~$N$, the tensor product of~$M$
and $H\otimes_K N$ is a left $H$-module
$M\otimes(H\otimes_K N)$. There is an isomorphism
of $H$-modules
\begin{equation}\label{eq:TwistingIso}
M\otimes(H\otimes_K N)
\xrightarrow[\iso]{\Theta}
H\otimes_K(M\otimes N);
\quad
m\otimes(h\otimes n)\mapsto
\sum_i h'_i\otimes(\bar{h''_i}m\otimes n)
\end{equation}
where $\bar{x}=\chi(x)$ and $M\otimes N$ is a left
$K$-module with the diagonal action.

A particular instance of this is
\begin{equation}\label{eq:H//K}
(H/\!/K)\otimes(H/\!/K) \iso H\otimes_K(H/\!/KB)
\end{equation}
where
\[
H/\!/K=H/HK^+\iso H\otimes_K\k.
\]

\begin{lem}\label{lem:TensorProd-extended}
Suppose that $H$ is coherent and $K$ is finite
dimensional. Let~$M$ be a left $H$-module and~$N$
a left $K$-module. Then the $H$-module
$M\otimes(H\otimes_KN)$ is pseudo-coherent.
\end{lem}
\begin{proof}
Using \eqref{eq:TwistingIso},
\[
M\otimes(H\otimes_KN) \iso H\otimes_K(M\otimes N)
\]
which is pseudo-coherent by Lemma~\ref{lem:pc-extended}.
\end{proof}

In general, the tensor product of two coherent
modules over a locally Frobenius Hopf algebra
is not a coherent module. However, it turns out
that it is pseudo-coherent.
\begin{prop}\label{prop:TensorProd-extended}
Suppose that $H$ is a locally Frobenius Hopf
algebra and that $L$ and $M$ are two coherent
left $H$-modules. Then the $H$-module
$L\otimes M$ is pseudo-coherent.
\end{prop}
\begin{proof}
Every coherent $H$-module is induced from
a finitely generated $H(\lambda)$-module for
some~$\lambda$. By choosing a large enough~$\lambda$
we can assume that $L\iso H\otimes_{H(\lambda)}L'$
for some finitely generated $H(\lambda)$-module~$L'$.
Then as $H$-modules,
\[
L\otimes M \iso H\otimes_{H(\lambda)}(L'\otimes M),
\]
which is a pseudo-coherent $H$-module.
\end{proof}

If $H$ is a locally Frobenius Hopf algebra,
then for~$\lambda$, the cyclic $H$-module
\[
H/\!/H(\lambda)=H/HH(\lambda)^+\iso H\otimes_{H(\lambda)}\k
\]
can be viewed as an $H(\lambda)$-module, and
by~\eqref{eq:H//K} there is an isomorphism of
$H$-modules
\[
H/\!/H(\lambda)\otimes H/\!/H(\lambda)
\iso H\otimes_{H(\lambda)}H/\!/H(\lambda).
\]
So if $H/\!/H(\lambda)$ is a coproduct of \fg./\fd.
$H(\lambda)$-modules then
$H/\!/H(\lambda)\otimes H/\!/H(\lambda)$ is
$\lambda$-strongly pseudo-coherent $H$-module.

Now we consider \emph{normality} for subHopf
algebras of locally Frobenius Hopf algebras.
\begin{defn}\label{defn:AdjAction}
The left and right \emph{adjoint actions}
of $h\in H$ on $x\in H$ are given by
\[
h\cdot x = \sum_i h'_ix\bar{h_i''},
\quad
x\cdot h = \sum_i \bar{h'_i}xh_i''.
\]
\end{defn}

When $H=\k G$ is a group algebra then
if $g\in G$, the adjoint action agrees
with conjugation by~$g$.

The adjoint action of $H$ on itself is
compatible with the product $\phi\:H\otimes H\to H$
and coproduct $\psi\:H\to H\otimes H$
provided $H\otimes H$ is given the
diagonal coactions for which
\[
h\cdot(x\otimes y) =
\sum_{i}h'_i\cdot x\otimes h''_i\cdot y,
\quad
(x\otimes y)\cdot h =
\sum_{i}x\cdot h'_i\otimes y\cdot h''_i.
\]

In the following definition, the notion of
\emph{strongly normalised} is motivated by
the case of the group algebra of a locally
finite group~$G$. If a finite subgroup
$H\leq G$ normalises a subgroup~$K$, then
for any finite subgroup $K'\leq K$, the
finite set
\[
\bigcup_{h\in H}hK'h^{-1}\subseteq K
\]
is contained in some finite subgroup
$K'\leq K$. This has implications for the
image of the adjoint action of $\k H$ on
$\k K'$ which is contained in $\k K''$.
For a locally Frobenius Hopf algebra our
definition specialises to this case.

\begin{defn}\label{defn:NormalsubHopfAlg}
A sub(Hopf) algebra $K\subseteq H$ is
\emph{normal} if for all $h\in H$,
$h\cdot K\subseteq K$ and $K\cdot h\subseteq K$.

If $H'\subseteq H$ is a subHopf algebra,
then~$K$ is \emph{normalised by $H'$} if
for every $h\in H'$, $h\cdot K\subseteq K$
and $K\cdot h\subseteq K$.

If $K\subseteq H$ is a locally Frobenius
subHopf algebra of shape $\Lambda'\subseteq\Lambda$
then~$K$ is \emph{strongly normalised} by
a subHopf algebra $H'\subseteq H$ if for
all~$h\in H'$ and $\lambda\in\Lambda'$,
there is a $\lambda_h\in\Lambda'$ such that
\[
h\cdot K(\lambda)\subseteq K(\lambda_h),
\quad K(\lambda)\cdot h \subseteq K(\lambda_h).
\]
\end{defn}

Notice that in the strongly normalised case,
if $H'$ is \fd., then
\[
\sum_{h\in H'}h\cdot K(\lambda)
\]
is a \fd. subspace so is contained in
some~$K(\tilde{\lambda}_h)$.

Here is an omnibus proposition combining
results found in
Montgomery~\cite{SM:HopfAlgActions}*{section~3.4}
\begin{prop}\label{prop:NormalsubHopfAlg}
Let $K\subseteq H$ be a subHopf algebra. \\
\emph{(a)} Let $K$ be normal. Then $HK^+=K^+H$
and this is a Hopf ideal in~$H$; furthermore,
the quotient homomorphism $H\to H/HK^+$ is
a morphism of Hopf algebras. \\
\emph{(b)} If $H$ is faithfully flat as a
left/right $K$-module and $HK^+=K^+H$,
then~$K$ is normal. \\
\emph{(c)} If $H$ is finite dimensional
then~$K$ is normal if and only if $HK^+=K^+H$.
\end{prop}

The next result is probably standard but
we do not know a convenient reference.
\begin{prop}\label{prop:semidirectprodHA}
Suppose that $H$ is a Hopf algebra over a
field and that $K,L$ are subHopf algebras
where~$K$ is normalised by~$L$. Then
$KL=LK$ and this is a subHopf algebra
of~$H$.
\end{prop}
\begin{proof}
This is similar to the proof of
Proposition~\ref{prop:NormalsubHopfAlg}(a).
The key point is that for $k\in K$ and
$\ell\in L$, using the definition of
the counit and coassociativity we
have,
\begin{align*}
\ell k
&= \sum_i \ell'_i k \epsilon(\ell''_i) \\
&= \sum_i \ell'_i\cdot k \ell''_i \in KL
\end{align*}
so $LK\subseteq KL$. Similarly, $KL\subseteq LK$.
\end{proof}
\begin{prop}\label{prop:normalsubFrobHA}
Let~$H$ be a locally Frobenius Hopf algebra
of shape~$\Lambda$ and let $K\subseteq H$
be a locally Frobenius subHopf algebra if
shape $\Lambda'\subseteq\Lambda$. Suppose
that $L\subseteq H$ is a finite dimensional
subHopf algebra and~$K$ is normalised by~$L$.
Then for any finite dimensional subHopf
algebra $K'\subseteq K$ there is a
$\lambda\in\Lambda$ such that $H(\lambda)$
contains~$K'$ and~$L$; furthermore
$K\cap H(\lambda)$ is normalised by~$L$ and
$L(K\cap H(\lambda))=(K\cap H(\lambda))L\subseteq H(\lambda)$
is a finite dimensional subHopf algebra
containing~$K'$ and~$L$.
\end{prop}
\begin{proof}
This is a straightforward consequence
of Lemma~\ref{lem:finitesubset} and
Proposition~\ref{prop:semidirectprodHA}.
\end{proof}

\subsection*{Locally finite $H$-modules and $H^\circ$-comodules}
For a locally Frobenius Hopf algebra $H$ the
finite dual $H^\circ$ is also a Hopf algebra.
It is standard that the categories of locally
finite $H$-modules $\Mod_H^{\mathrm{l.f.}}$
and left $H^\circ$-comodules $\Comod_{H^\circ}$
are equivalent (in fact isomorphic). The latter
category has as injectives summands of extended
comodules $H^\circ\otimes W$, so when viewed as
locally finite $H$-modules these are the injectives
of~$\Mod_H^{\mathrm{l.f.}}$. It is also well known
that $\Comod_{H^\circ}$ and $\Mod_H^{\mathrm{l.f.}}$
lack projectives so only the right derived functors
of
\[
\Hom_H(M,-)\iso\Cohom_{H^\circ}(M,-)=\Comod_{H^\circ}(M,-)
\]
are defined and $\Coext_{H^\circ}(-,-)$ is not
a balanced functor. We study the graded analogue
of this for $P$-algebras in~\cite{AB:Palgebras}.

Of course the existence of injectives makes it
possible to define a stable module category
using these. The tensor product of two locally
finite $H$-modules is also locally finite so
$\Mod_H^{\mathrm{l.f.}}$ has a monoidal
structure and this passes to the stable module
category since for any $H^\circ$-comodule~$M$,
$H^\circ\otimes M$ is isomorphic to an extended
comodule.

%

\section{Some examples of locally Frobenius
Hopf algebras}\label{sec:Examples}

In this section we decribe some examples
which occur in the literature whose
understanding might be aided by viewing
them as locally Frobenius algebras but
we leave detailed investigation for
future work.

\subsection{Group algebras of locally finite groups}
\label{subsec:LocFinGpAlg}

Recall that a countable group $G$ is \emph{locally
finite} if every finite subset $S\subseteq G$
is contained in a finite subgroup. The group
algebras of such groups have been studied in
the literature, for example in the work of
Hartley, Richardson and Musson
~\cite{BH&JSR:SocleGpRngs,JSR:GpRngsnonzerosocle,IMM:InjModGpAlgLOcFinGps}.
\begin{prop}\label{prop:LocFinGp}
Let\/ $\k$ be a field and $G$ a locally
finite group. Then\/~$\k G$ is a locally
Frobenius Hopf algebra.
\end{prop}
\begin{proof}
We take the indexing set $\lambda$ to be
the set of finite subgroups ordered by
inclusion and for each $H\in\lambda$ take
the finite group algebras~$\k H$ which is
a subHopf algebra of~$\k G$.

A finite dimensional subspace $V\subseteq\k G$
has a basis whose elements are expressible
as linear combinations of finitely many
elements of~$G$, therefore it must be
contained in some finite subgroup~$H\leq G$,
so~$V\subseteq\k H$.

For two finite subgroups $H\leq K\leq G$,
there are left and right integrals
\[
\sum_{h\in H}h\in\sint_{\k H},
\quad
\sum_{k\in K}h\in\sint_{\k K}.
\]
Then taking a complete set of right
coset representatives
$k_1,\ldots,k_{|K:H|}$ for $K/H$ we
have
\[
\biggl(\k_1+\cdots+k_{|K:H|}\biggr)
\sum_{h\in H}h
= \sum_{k\in K}k,
\]
so $\sint_{\k K}\subseteq\k K\sint_{\k H}$.
The condition of
Lorenz~\cite{ML:TourRepThy}*{~12.4.1(a)}
shows that $\k K:\k H$ is a Frobenius
extension.
\end{proof}

\subsection{Dual profinite group algebras}
\label{subsec:ProfiniteGp}

Let $G$ be a profinite group,
\[
G = \lim_{\substack{N\lhd G\\|G:N|<\infty}}G/N.
\]
For a field $\k$ the pro-group algebra
\[
\k G = \lim_{\substack{N\lhd G\\|G:N|<\infty}}\k G/N
\]
is a complete topological Hopf algebra.
However, there is also dual object
\[
\k(G) =
\colim_{\substack{N\lhd G\\|G:N|<\infty}}
\k(G/N)
\]
where
\[
\k(G/N) = \Map(G/N,\k) \iso \Hom_\k(\k G/N,\k)
\]
is the dual group ring. Alternatively,~$\k(G)$
agrees with the algebra of locally constant
(continuous) functions $G\to\k$ with respect
to the profinite topology. Since each $\k(G/N)$
is a commutative Hopf algebra so is~$\k(G)$.
In fact, $\k(G)$ also agrees with the finite
dual of~$\k G$, commonly denoted by
$(\k G)^{\mathrm{o}}$.

\begin{prop}\label{prop:ProfinGp}
For a profinite group $G$,\/~$\k(G)$ is a
commutative locally Frobenius Hopf algebra.
\end{prop}
\begin{proof}
We need to check that for cofinite normal
subgroups~$M\lhd G$ and~$N\lhd G$ with~$M\lhd N$,
$\k(G/M):\k(G/N)$ is a Frobenius extension.

The functions $\delta_{1_{G/M}}\in\k(G/M)$
and $\delta_{1_{G/N}}\in\k(G/N)$ with
\begin{align*}
\delta_{1_{G/M}}(gM)
&=
\begin{dcases*}
1 & if $g\in M$, \\
0 & otherwise,
\end{dcases*}    \\
\delta_{1_{G/N}}(gN)
&=
\begin{dcases*}
1 & if $g\in N$, \\
0 & otherwise,
\end{dcases*}
\end{align*}
are integrals for $\k(G/M)$ and $\k(G/N)$.
The quotient homomorphism $\pi\:G/M\to G/N$
satisfies
\[
\pi^*\delta_{1_{G/N}} =
\sum_{gM\in\ker\pi}\delta_{gM}
\]
and
\[
\delta_{1_{G/M}}\pi^*\delta_{1_{G/N}}
= \delta_{1_{G/M}},
\]
so the condition of Lorenz~\cite{ML:TourRepThy}*{~12.4.1(a)}
shows that $\k(G/M):\k(G/N)$ is a Frobenius
extension.
%
\end{proof}

\end{document}